\documentclass[a4paper,10pt]{amsart}
\usepackage[utf8]{inputenc}
\usepackage{mathrsfs}
\usepackage{amsmath}
\usepackage{amssymb}
\usepackage[all]{xy}
\usepackage{pigpen}
\title[Finite presheaves]{Finite presheaves and $\cala$-finite generation of unstable algebras mod nilpotents}
\author{Geoffrey Powell}
\address{Laboratoire angevin de recherches en mathématiques (LAREMA),
CNRS, Université d’Angers, Université Bretagne Loire, 
2 Bd lavoisier 49045 Angers Cedex 01
}
\email{Geoffrey.Powell@math.cnrs.fr}
\urladdr{http://math.univ-angers.fr/~powell/}
\thanks{These results were presented at the Conference in honour of Lionel 
Schwartz at Paris in June 2017. This work was partially supported by the ANR 
Project {\em ChroK}, 
{\tt ANR-16-CE40-0003}.
}

\newtheorem{THM}{Theorem}

\newtheorem{thm}{Theorem}[subsection]
\newtheorem{prop}[thm]{Proposition}
\newtheorem{cor}[thm]{Corollary}
\newtheorem{lem}[thm]{Lemma}
\theoremstyle{definition}
\newtheorem{defn}[thm]{Definition}
\newtheorem{exam}[thm]{Example}

\theoremstyle{remark}
\newtheorem{rem}[thm]{Remark}
\newtheorem{nota}[thm]{Notation}

\renewcommand{\phi}{\varphi}
\renewcommand{\epsilon}{\varepsilon}
\renewcommand{\hom}{\mathrm{Hom}}
\newcommand{\po}{\ar@{}[dr]|(.7){\text{\pigpenfont R}}}
\newcommand{\pb}{\ar@{}[dr]|(.3){\text{\pigpenfont J}}}
\newcommand{\ob}{\mathrm{ob}}
\newcommand{\g}{\mathfrak{g}}
\newcommand{\gtop}{\g_{\mathrm{top}}}
\newcommand{\unstalg}{\mathscr{K}}
\newcommand{\unstalgc}{\mathscr{K}^+}
\newcommand{\ffin}[1][n]{\f_{#1}^{\mathrm{fin}}}

\newcommand{\f}{\mathscr{F}}
\newcommand{\unst}{\mathscr{U}}
\newcommand{\fdvs}{\vs_f}
\newcommand{\vs}{\mathscr{V}}

\newcommand{\pfset}{\mathscr{P}\hspace{-2pt}\mathscr{S}}
\newcommand{\nil}{\mathscr{N}il}
\newcommand{\op}{^\mathrm{op}}
\newcommand{\field}{\mathbb{F}}
\newcommand{\fieldp}{{\mathbb{F}}}
\newcommand{\fieldt}{{\mathbb{F}}}
\newcommand{\prshv}{\widehat{\fdvs}}
\newcommand{\pvc}{(\prshv)_c}

\newcommand{\nat}{\mathbb{N}}
\newcommand{\cala}{\mathscr{A}}
\newcommand{\gfin}{\prshv^{\mathrm{fin}}}

\newcommand{\gps}{\prshv^{\mathrm{profin}}}
\newcommand{\gpsc}{(\gps)_c}
\newcommand{\gcoan}{\prshv^{\omega}}
\newcommand{\End}{\mathrm{End}}
\newcommand{\reg}{\mathrm{reg}}
\newcommand{\Aut}{\mathrm{Aut}}
\newcommand{\Surj}{\mathrm{Surj}}
\newcommand{\grass}{\mathrm{Gr}}
\newcommand{\cali}{\mathscr{I}}

\newcommand{\hopfk}{\mathscr{H}_\mathscr{K}}
\newcommand{\hopfkc}{\hopfk^+}
\newcommand{\spec}{\mathrm{Spec}}
\newcommand{\gl}{\mathrm{GL}}

\newcommand{\reals}{\mathbb{R}}
\newcommand{\calc}{\mathscr{C}}
\newcommand{\grp}{\mathsf{Gp}}
\newcommand{\fpg}{\mathscr{G}^f_p}
\newcommand{\pgfin}{\mathscr{G}^{\mathrm{fin}}_p}
\newcommand{\pgpro}{\mathscr{G}^{\omega +}_p}

\newcommand{\hq}{/\hspace{-3pt}/}
\newcommand{\dash}{\hspace{-3pt}-\hspace{-3pt}}
\newcommand{\cre}{\mathsf{cr}}

\begin{document}

\begin{abstract}
 Inspired by the work of Henn, Lannes and Schwartz on unstable algebras over the 
Steenrod algebra modulo nilpotents, 
  a characterization of unstable algebras that are $\cala$-finitely generated up 
to nilpotents is given in terms 
  of the associated presheaf, by introducing the notion of a finite presheaf. In 
particular, this gives the natural 
  characterization of the (co)analytic presheaves that are important in the 
theory of Henn, Lannes and Schwartz.   An important source of examples 
is  provided by unstable algebras of finite transcendence degree. 
  
  For unstable Hopf algebras, it is shown that the associated 
presheaf is finite if and only 
  if its growth function is polynomial. This leads to a description of unstable 
Hopf algebras modulo nilpotents in the spirit 
  of Henn, Lannes and Schwartz.
\end{abstract}

\maketitle

\section{Introduction}

The work of Henn, Lannes and Schwartz \cite{HLS} explains the 
relationship between 
unstable algebras over the mod $p$ Steenrod algebra $\cala$ and presheaves of 
profinite sets on the category 
$\fdvs$ of finite-dimensional $\field_p$-vector spaces, where $\field_p$ is the prime field. This is based upon 
Lannes' $T$-functor; namely, for $K$ an unstable algebra,  
there is a  presheaf $\g K$ given by $V \mapsto \hom_\unstalg (K, H^* (BV; 
\field_p))$, $V$ an $\field_p$-vector space, where $\unstalg$ is the category 
of 
unstable algebras.  
Motivating examples of unstable algebras are given by singular cohomology $H^* 
(Y; 
\field_p)$, for $Y$ a topological space.  Another presheaf is 
provided by homotopy classes of maps out of $BV$, $V \mapsto [BV, Y]$. These 
presheaves are related by Lannes' theory, 
 notably in relation to the Sullivan conjecture. 

One of the key results of \cite[Part II]{HLS} gives a characterization for an 
unstable algebra $K$ to be Noetherian up to nilpotents
 in terms of the associated presheaf $\g K$ (this is recalled here in Section 
\ref{subsect:HLS_trans_deg}). Likewise, the notion of transcendence degree 
 is treated in terms of $\g K$. The relevance of these  to the current 
theory is explained below.

 In the context of unstable algebras, the more general notion of being 
$\cala$-finitely generated (see Section \ref{subsect:unstalg_prelim}) is known 
to be 
 of importance. For example, for an $H$-space, Castellana, Crespo and Scherer 
\cite{CCS} showed that this condition imposes strong conditions on its 
structure as an $H$-space. 
 One of the main objectives of this paper is to relax the condition of being 
$\cala$-finitely generated in the spirit of Henn, Lannes and Schwartz, by 
considering unstable algebras
  that are $\cala$-finitely generated up to nilpotents, and to give a 
characterization in terms of the associated presheaf.

This is carried out in the first part of the paper, by introducing the notion of a finite presheaf. This is given in  terms 
of presheaves of $\field_p$-vector spaces. The latter form an abelian category  
and such a presheaf is said to be  finite if it has a finite composition 
series. 
A set-valued presheaf $X$ is defined to be finite if there exists an embedding 
 $X \hookrightarrow F_X$ with $F_X$ a finite presheaf of $\field_p$-vector 
spaces. 
 
 This definition may seem somewhat {\em ad hoc} at first view, in particular 
due 
to its dependence on 
using  presheaves of $\field_p$-vector spaces.  As shown in the second part of 
the paper,  
this condition on $F_X$  can be relaxed:  it suffices 
that $F_X$ 
  be a presheaf of finite $p$-groups that is a polynomial functor in the 
sense of Baues and Pirashvili  \cite{BP} (generalizing the Eilenberg-MacLane 
notion of a polynomial functor to an abelian category \cite{EM}, as  
recalled in Appendix \ref{sect:poly}); 
this is a consequence of Theorem 
  \ref{thm:p-finite_equiv}. 
 
 The characterization of $\cala$-finitely generated up to nilpotents is as 
follows:

 \begin{THM}
[Corollary \ref{cor:equivalent_finite_fgmodnil}]
For $K$ an unstable algebra, the following are equivalent:
\begin{enumerate}
 \item 
 $K$ is $\cala$-finitely generated up to nilpotents;
 \item 
 $\g K$ is a finite presheaf.
\end{enumerate}
  \end{THM}

  This allows the theory of \cite{HLS} to be revisited and made slightly more 
precise (see Theorem \ref{thm:HLS_improved}).

The class of finite presheaves is, as of yet, imperfectly understood; a reassuring general fact
is the following, for unstable algebras of finite transcendence degree:
  
  \begin{THM}
   [Corollary \ref{cor:trans_deg_finite_presheaf}]
   Let $K$ be an unstable algebra of finite transcendence degree. The following 
are equivalent:
   \begin{enumerate}
    \item 
     $K$ is $\cala$-finitely generated up to nilpotents;
     \item 
     $\g K$ takes values in finite sets.
   \end{enumerate}
  \end{THM}

In general it is not easy to determine whether a given presheaf 
$X$ taking values in finite sets is a finite presheaf.    A necessary condition is that the 
associated growth function defined on 
$\nat$ by 
 $    \gamma_X (t):= \log_p | X (\field_p^t)|
   $
should have polynomial growth. As illustrated by Example \ref{exam:split_rank}, 
this 
is not sufficient. 

Such growth functions occur in the work of Lannes and Schwartz \cite{LS} and of 
Grodal \cite{G}, notably in relation to the study of finite Postnikov systems, 
 for which (under suitable hypotheses), the growth functions are shown to 
have polynomial growth. Indeed, revisiting these results in the light of 
subsequent 
developments of the theory provided one of the motivations for introducing the 
notion of a finite presheaf.

Even in the case of two-stage Postnikov systems, there 
 remain basic open questions such as:  for which two-stage Postnikov systems $Y$  is $\g H^* 
(Y; \field_p)$ a finite presheaf?

 When $Y$ is an $H$-space, then the answer is yes, tying in with the work of 
\cite{CCS}. Indeed, in general in this theory, the 
situation is much better when one restricts to 
unstable Hopf algebras (for these, see Section \ref{sect:hopf}). This is the 
subject of the second part of the paper. 

For instance, one has the following, which should be of independent interest:

\begin{THM}
[Theorem \ref{thm:p-finite_equiv} and Theorem \ref{thm:pfin_finite}]
For $H$ a connected unstable Hopf algebra over $\field_p$, the following are equivalent:
\begin{enumerate}
 \item 
 the underlying unstable algebra of $H$ is $\cala$-finitely generated up to 
nilpotents;
 \item 
 $\g H$ is a finite presheaf;
 \item 
 $\g H$ takes values in finite $p$-groups and is polynomial;
 \item 
 $\gamma_{\g H}$ has polynomial growth. 
\end{enumerate}
\end{THM}

As a consequence, a Henn-Lannes-Schwartz style characterization of connected unstable Hopf algebras 
up to nilpotents is given (see Theorem \ref{thm:g_Hopf_fun-pfg}). 

\part{Finite presheaves}
\section{Dramatis Person\ae}
\label{sect:personae}

This section introduces the presheaf categories that are central to the paper 
and recalls the notion 
of finite and polynomial abelian presheaves. The relationship with unstable 
algebras is explained, in 
particular recalling some of the salient points of Henn-Lannes-Schwartz theory 
\cite{HLS}.

\subsection{Presheaves on $\fieldp$-vector spaces}

Fix a prime $p$ and let $\fdvs \subset \vs$ denote the full subcategory of 
finite-dimensional spaces in the category $\vs$ of $\fieldp$-vector spaces, where $\fieldp$ denotes the prime field of characteristic $p$. 

\begin{nota}
Denote by  
 \begin{enumerate}
  \item 
  $\prshv$ the category of presheaves of sets on $\fdvs$ (i.e. contravariant 
functors 
from $\fdvs$ to  sets);
\item 
$\gps$ the category of presheaves of profinite sets on $\fdvs$, equipped with 
the forgetful functor $\gps \rightarrow \prshv$;
  \item 
  $\pvc \subset \prshv$ (respectively $\gpsc \subset \gps$) the full 
subcategory of connected 
objects, namely $X$ such that $|X(0)|=1$;
  \item 
  $\f$ the category of functors from $\fdvs \op$ to $\vs$, with abelian 
structure inherited from $\vs$.
 \end{enumerate}
\end{nota}

\begin{rem}
In the literature (eg. \cite{HLS,KI}), $\f$ usually denotes {\em 
covariant} functors. However, 
since vector space duality $V \mapsto V^\sharp$ restricts to an equivalence of 
categories 
$\fdvs \op \cong \fdvs$, the choice of variance is of 
little import. In particular, to simplify notation,  for $F$ a covariant 
functor, the associated contravariant functor $V \mapsto F (V^\sharp)$ will 
again be 
denoted by $F$.
\end{rem}

The following is standard:
 
\begin{lem}
\label{lem:split-reduced}
For $F \in \ob \f$, there is a canonical direct sum decomposition 
$
 F \cong F(0) \oplus \overline{F}, 
$
where $F(0)$ is considered as a constant functor and $\overline{F}$ is constant-free (i.e. $\overline{F}(0)=0$).
\end{lem}

The following results are  clear.

\begin{lem}
\label{lem:prshv_prod_coprod}
For $\prshv$, 
\begin{enumerate}
 \item 
  the coproduct $\amalg$ and product $\prod$ are inherited from the 
category of sets;
\item 
the product restricts to the product of $\pvc$;
\item 
 the coproduct of $\pvc$ is given by the  wedge product $\vee$.
\end{enumerate}
\end{lem}

There is a forgetful functor $\f \rightarrow \prshv$ that retains only the 
underlying set 
of a vector space. For $F \in \ob \f$, the 
underlying presheaf lies in $\pvc$ if and only if  $F$ is constant-free.

\begin{lem}
\label{lem:adj_F_linearization}
 The forgetful functor $\f \rightarrow \prshv$ admits as left adjoint, $X 
\mapsto \fieldp [X]$, so that there is an adjunction 
\[
 \fieldp [-] :  \prshv \rightleftarrows \f .
\]
The adjunction unit is the natural inclusion $X \hookrightarrow \fieldp [X]$. 

Hence, the forgetful functor preserves products: in particular, the underlying 
presheaf of $F \oplus G$ (for $F,G \in \ob \f$) is $F \times G$.
 \end{lem}

\begin{nota}
 For $X \in \ob \prshv$ and $x \in X(0)$, let $X_x \subset X$ denote the  
connected presheaf 
with $X_x (V)$ the fibre over $x$ of the surjection $X(V) \rightarrow 
X(0)$ induced by  $0 \hookrightarrow V$.
\end{nota}

\begin{lem}
\label{lem:connected components}
 For $X \in \ob \prshv$,
\begin{enumerate}
 \item  
 there is a natural isomorphism $X \cong \amalg_{x \in 
X(0)} X_x$;
\item 
if  $X \in \ob \pvc$ and $Y \in \ob \prshv$, 
 there is a natural bijection
 \[
  \hom_{\prshv} (X, Y) 
  \cong 
  \amalg_{y \in Y(0)} \hom_{\pvc} (X, Y_y).
 \]
\end{enumerate}
\end{lem}

\begin{proof}
The first statement is clear. For the second, for a map of presheaves $f : X 
\rightarrow Y$, the image $f (x) \in Y(0)$ of 
the 
basepoint $x\in X(0)$ determines the connected component of the image of $f$.
\end{proof}

\subsection{Polynomial functors}

The following general definition applies for example to the category $\f$:
 
 \begin{defn}
 An object $F$ of an abelian category is finite if it has a finite composition 
series.  
 \end{defn}

The Eilenberg-MacLane definition of a polynomial functor \cite{EM} (see Appendix 
 \ref{sect:poly}, where a general definition is given that covers the abelian 
case) 
also applies to $\f$, 
and one 
has:

\begin{prop}
\label{prop:finite_polynomial}
 \cite{KI}
 A functor $F \in \ob \f$ is finite if and only if it is polynomial and takes 
values in $\fdvs$. 
\end{prop}

\begin{exam}
 The $n$th symmetric power functor $S^n$ and the $n$th exterior power functor 
$\Lambda^n$ are both 
 polynomial functors in $\f$ of degree $n$.
\end{exam}

\begin{rem}
 The polynomial degree of a functor $F \in \ob \f$ that takes 
finite-dimensional 
values can be characterized in terms of 
 its growth function (see Proposition \ref{prop:poly_degree_growth}). 
\end{rem}

The inclusion of the full subcategory of functors of $\f$ of 
Eilenberg-MacLane polynomial degree at most $n \in \nat$  
admits a left adjoint 
$q_n$ (which is considered as a functor $q_n : \f \rightarrow \f$), so that the 
adjunction unit 
 $
 F \twoheadrightarrow q_n F
 $ 
is the universal map to a polynomial functor of degree at most $n$. (See 
\cite[Section 4.2]{K_krull} and the references therein.)

\subsection{Unstable algebras}
\label{subsect:unstalg_prelim}

One motivation for considering profinite presheaves  comes from the study of 
the 
category $\unstalg$ of unstable algebras over the mod $p$ Steenrod algebra 
$\cala$ \cite{HLS}. 

As usual, the category of unstable modules over $\cala$ is denoted $\unst$ and 
the Steenrod-Epstein (or Massey-Peterson) enveloping algebra 
functor by $U : \unst \rightarrow \unstalg$. The reader is referred to  
\cite{S} for the basic theory of unstable modules and algebras, together with 
 the localization of $\unst$ away from $\nil$, the subcategory of nilpotent 
unstable modules, and the associated localization functor 
 $\unst \rightarrow \unst / \nil$. The theory and applications of this 
nillocalization have been developed in \cite[Part 
I]{HLS}, \cite{KI} and in subsequent work by many authors.

In \cite[Part II]{HLS}, the analogous (non-abelian) theory for unstable 
algebras was developed, leading to the construction of the localized category 
$\unstalg / \nil$. Recall the  terminology introduced by Quillen: 

\begin{defn}\cite[Part II]{HLS}
 A morphism $\phi : K \rightarrow L$ in $\unstalg$ is
\begin{enumerate}
 \item 
an $F$-monomorphism if, $\forall x \in \ker \phi$, $\exists n \in \nat $ such 
that $x^n=0$;
\item 
an $F$-epimorphism  if, $\forall y \in L$, $\exists m \in \nat $ such that 
$y^{p^m} \in 
\mathrm{image}\  \phi$;
\item 
an $F$-isomorphism if $\phi$ satisfies both the above.
\end{enumerate}
\end{defn}

The localization $\unstalg \rightarrow \unstalg / \nil$ is universal amongst 
functors $\unstalg \rightarrow \mathscr{C}$ which send $F$-isomorphisms to 
isomorphisms of $\mathscr{C}$.

\begin{defn}
\label{def:A_fg}
An unstable algebra $K \in \ob \unstalg$ is $\cala$-finitely generated if there exists a finitely generated unstable module $M$ and a morphism of unstable modules $M \rightarrow K$ such that the 
 induced map of unstable algebras $UM \rightarrow K$  is surjective.
\end{defn}

\begin{nota}
 Let $\unstalgc$ denote the full subcategory of connected unstable algebras 
($K$ such that $K^0= \fieldp$).
\end{nota}

\begin{rem}
\label{rem:cala_fg_conn}
A connected unstable algebra $K \in \ob  \unstalgc$ is canonically augmented, hence the 
augmentation ideal $\overline{K}$ and the module of indecomposables, $QK:= 
\overline{K}/ \overline{K}^2$, are defined.

In this case, $K$ is $\cala$-finitely generated if and only if  $QK$ is a finitely generated $\cala$-module.
\end{rem}

\begin{rem}
\label{rem:CCS}
The condition that the cohomology of a space is $\cala$-finitely 
generated imposes strong conditions, notably when working with $H$-spaces;    see \cite[Theorem 7.3]{CCS} for example.
\end{rem}

\subsection{From unstable algebras to presheaves}

The following is clear:

\begin{prop}
\label{prop:colimit_Afg}
Every unstable algebra $K$ is the colimit of its $\cala$-finitely generated sub 
unstable algebras.
\end{prop}

\begin{defn}
Let 
\begin{enumerate}
 \item 
 $ \g : \unstalg \op \rightarrow \gps 
$ be the functor defined by $\g K (V) := \hom_{\unstalg} (K, H^* (BV))$, where 
$H^* (BV)$ denotes the mod $p$ group cohomology of $V$;
\item 
$
  \g : (\unstalgc) \op \rightarrow \gpsc
$ denote the restriction to connected objects. 
\end{enumerate}
Here the profinite structure  arises from Proposition \ref{prop:colimit_Afg}.
\end{defn}

The following is a fundamental fact, following from the analogous result for 
$\unst$ (as in \cite[Part I]{HLS} and \cite{KI}), together with Lannes' 
linearization principle. 

\begin{prop}
 \cite{HLS}
 \label{prop:g_factors_nil}
 The functor $\g$ factors naturally 
  $$
  \g : (\unstalg / \nil) \op \rightarrow \gps.
  $$
\end{prop}

\begin{nota}
 For $Y$ a topological space, let $\gtop Y \in \ob \prshv$ denote the presheaf 
$
  \gtop Y (V):= [BV, Y].
 $
\end{nota}

The presheaves $\g$ and $\gtop$ are intimately related (compare \cite{LS}):

\begin{thm}
\label{thm:Lannes_g_gtop}
\cite{La}
 For $Y$ a topological space, mod $p$ cohomology induces a morphism of 
presheaves 
 $
  \gtop Y \rightarrow 
  \g H^* Y
 $
that is an isomorphism if $Y$ is connected, nilpotent, $\pi_1 Y$ is finite and 
$H^* (Y)$ is of finite type.
\end{thm}

\subsection{From presheaves to unstable algebras}

For simplicity of presentation, suppose that $p=2$ in this section.  The odd 
primary case is treated by passage to objects concentrated in even 
degrees 
(cf. \cite{HLS}, for example).

\begin{defn}
 Let $\kappa : (\gps) \op \rightarrow \unstalg$, the associated unstable 
algebra 
functor, 
 be defined by 
 \[
  \kappa : X \mapsto \hom_{\gps} (X, S^*)
 \]
where the commutative $\fieldt$-algebra structure of $\kappa X$ is induced by 
the  commutative algebra structure of $S^* $ and 
 Steenrod operations act via  $\hom_\f (S^* , S^*)$
 (cf. \cite{KI}).
\end{defn}

\begin{exam}
\label{exam:kappa_Sn} For $n \in \nat$, $\kappa S^n \cong U F(n)$, the free 
unstable algebra on a generator of degree $n$ (see \cite{K_comp}), thus $\kappa S^n \cong H^* (K (\fieldt,n); \fieldt)$. 
\end{exam}

To stress the relationship between presheaves and unstable algebras, recall the
following part of \cite[Theorem II.1.5]{HLS}:

\begin{prop}
\label{prop:kappa_g}
 For $K \in \ob \unstalg$, there is a natural transformation
 $
  K 
  \rightarrow 
  \kappa \g K
 $
that is an $F$-isomorphism.
\end{prop}

The following  is recorded for later use:

\begin{prop}
 \label{prop:finite-type_kappa}
 Let $X \in \prshv$ take values in finite sets. Then the unstable algebra 
$\kappa X$ has finite type.
\end{prop}

\begin{proof}
In degree $n$, $(\kappa X)^n = \hom _{\prshv} (X, S^n) \cong \hom_{\f} (q_n 
\fieldt [X], S^n)$.
  Now, by construction, $q_n \fieldt [X]$ is a polynomial functor which takes 
finite-dimensional values, hence is finite, by Proposition 
\ref{prop:finite_polynomial}; likewise, $S^n$ is finite. It 
follows that $\hom_{\f} (q_n \fieldt[X], S^n)$ is a finite-dimensional 
  vector space, whence the result.
\end{proof}

\subsection{Finite transcendence degree and Noetherian unstable algebras}
\label{subsect:HLS_trans_deg}

Recall (see \cite[Section II.2]{HLS}) that the transcendence degree of an 
unstable algebra $K$ is the 
transcendence degree of its underlying graded algebra (namely the supremum of 
the cardinalities of finite subsets of algebraically independent homogeneous 
elements of $K$). Transcendence degree is invariant under 
$F$-isomorphism.

\begin{defn}
 For $d \in \nat$, let $\unstalg_d$ be the full subcategory of unstable 
algebras of transcendence degree at most $d$ and  $\unstalg_d / \nil$  
the corresponding full subcategory of $\unstalg/ \nil$. 
\end{defn}

\begin{nota}
 For $d \in \nat$, let $\pfset\dash\End(\field^d_p)$ denote the category of 
profinite 
right $\End(\fieldp^d)$-sets.
\end{nota}

\begin{thm}
\label{thm:HLS_trans_deg}
 \cite[Theorems II.2.7, II.2.8]{HLS}
 For $d \in \nat$,  
 the functor $\g$ induces an equivalence of categories
  $
  (\unstalg_d/\nil) \op \cong \pfset\dash\End(\fieldp^d)
 $ 
with inverse $\kappa_d$ induced by $\kappa$.
 \end{thm}

Henn, Lannes and Schwartz \cite[Definition II.5.8]{HLS} also introduce 
the notion of a Noetherian $\End (\fieldp^d)$-set (necessarily finite) which 
allows them to give a characterization of 
 unstable algebras that are Noetherian up to nilpotents:
 
 \begin{thm}
 \label{thm:HLS_noetherian}
  \cite[Theorem II.7.1]{HLS}
  Let $d \in \nat$. 
  \begin{enumerate}
   \item 
   For $K$ a Noetherian unstable algebra, $\g_d K$ is a Noetherian $\End 
(\fieldp^d)$-set. 
   \item 
   If $S$ is a Noetherian $\End (\fieldp^d)$-set, then $\kappa_d S$ is a 
Noetherian unstable algebra of transcendence degree at most $d$. 
  \end{enumerate} 
 \end{thm}

 In particular, this leads to the following:
 
 \begin{defn}
  \label{def:noetherian_mod_nil}
  An unstable algebra $K$ is  Noetherian up to nilpotents if it has finite 
transcendence degree and, for any $d \in \nat$,  $\g_d K$ is a Noetherian $\End 
(\fieldp^d)$-set. 
 \end{defn}

\section{Finite presheaves and coanalyticity}
\label{sect:defs}

This section introduces the notion of a finite presheaf. Its relevance 
for 
unstable algebras is explained in Section \ref{sect:unstable}.

\subsection{Definitions and first properties}

\begin{defn}
\label{def:finite}
  An object $X$ of $\prshv$ is finite if there exists a finite functor $F_X \in 
\ob \f$ and 
  a monomorphism 
  \[
   X \hookrightarrow F_X
  \]
in $\prshv$. The full subcategory  of finite presheaves is denoted 
$\gfin$.
\end{defn}

\begin{lem}
\label{lem:finite2}
 A presheaf $X$ of $\prshv$ is finite if and only if  $|X(0)| < \infty$ and, for each $x \in 
X(0)$, there exists a constant-free finite functor $F_x \in \f$ and a 
monomorphism 
 $
  X_x \hookrightarrow F_x. 
 $
 \end{lem}

\begin{proof}
First suppose that $X$ is a finite presheaf, so that there exists  $X \hookrightarrow F_X$ where  $F_X \in \ob 
\f$ is a finite functor. Then $X(0)$ is a finite set and, 
for any $x \in X$, the map 
\[
 X_x \hookrightarrow X \hookrightarrow F_X \twoheadrightarrow \overline{F_X} 
\]
is a monomorphism (by Lemma \ref{lem:connected components}). 

Conversely, suppose  that $X(0)$ is a finite set and that there exist 
injective 
maps $f_x : X_x \hookrightarrow 
F_x$, $\forall x \in X(0)$, with $F_x \in \ob \f$ finite and constant-free. Then the 
map 
\[
 X \hookrightarrow F_X :=  \fieldp [X(0)] \oplus \bigoplus_{x \in X(0)}  F_x 
\]
defined on $X_x$ by the constant map $X_x \rightarrow \fieldp[X(0)]$ to $[x]$, 
the map $f_x$ to $F_x$ and the zero map to the  components $F_y$, $y \neq 
x$, is an injection into a finite functor of $\f$, exhibiting $X$ as a finite functor.
\end{proof}
 
\begin{rem}
The definition of a finite presheaf extends verbatim to $\gps$. This leads to 
no increased generality, since a finite presheaf necessarily takes values in 
 finite sets.
\end{rem}

\begin{lem}
\label{lem:elementary-prop-finite}
 Let $X, Y \in \ob \prshv$ be finite presheaves. Then 
 \begin{enumerate}
  \item 
  $X \amalg Y$ is finite; 
  \item 
  $X \times Y$ is finite; 
  \item 
  if $U \subset X$ is a subobject, then $U$ is finite; 
  \item 
  if $X, Y$ are connected, with respective basepoints $x, y$, then $X \vee Y$ 
is finite.
\end{enumerate}
\end{lem}

\begin{proof}
 The case of $X \amalg Y$ follows easily from Lemma  \ref{lem:finite2}. For 
$X \times Y$, 
 the inclusions $X\hookrightarrow F_X$ and $Y \hookrightarrow F_Y$, with $F_X, 
F_Y \in \ob \f$ finite, induce  $X \times Y \hookrightarrow F_X \oplus F_Y$ by 
cartesian product. 
 
 The preservation of finiteness under passage to subobjects is clear. 
 Considering $X\vee Y$  as the subobject of $X \times Y$  given by $X \times 
y \cup x \times Y$ gives the final statement.
\end{proof}

\subsection{The degree of a finite presheaf}

\begin{prop}
\label{prop:degree_prshv}
 For $X \in \ob \gps$, the following conditions are equivalent:
 \begin{enumerate}
  \item 
  $X$ is finite;
  \item 
  $X$ takes values in finite sets and there exists $n \in \nat$ such that the 
composite 
  $
   X \hookrightarrow \fieldp[X] \twoheadrightarrow q_n \fieldp [X] 
  $
is a monomorphism.
 \end{enumerate}
The degree of  a finite $X$ is the least such $n$.
\end{prop}

\begin{proof}
 The condition that $X$ takes finite values is necessary; under this 
hypothesis, for any $n$, $q_n \fieldp [X]$ is a finite functor (by Proposition 
\ref{prop:finite_polynomial}), hence the 
existence of a monomorphism $X \hookrightarrow q_n \fieldp [X]$ implies that $X$ 
is finite. 
 
 Conversely, suppose that $X$ is finite, so that there is a finite functor 
$F_X$ and an inclusion $X \hookrightarrow F_X$. Let $n$ be the polynomial 
degree of $F_X$, then the induced linear map $\fieldp[X] \rightarrow F_X$ 
(provided by Lemma \ref{lem:adj_F_linearization}) factors 
across $q_n \fieldp[X]$, from which the result follows.  
\end{proof}

The relevance of the functor $q_n \circ \fieldp [-]$ is explained by the 
following straightforward  proposition, using:

\begin{nota}
 For $n \in \nat$, denote by 
 \begin{enumerate}
  \item 
   $\ffin \subset \f$ the full subcategory with objects 
 finite functors of polynomial degree at most $n$;
 \item 
 $(\gfin)_n \subset \prshv$ the full subcategory of finite presheaves of 
degree at most $n$.
 \end{enumerate}
\end{nota}

\begin{prop}
\label{prop:q_nF_ajdunction}
 For $n \in \nat$, the forgetful functor $\f \rightarrow \prshv$ restricts to 
$\ffin \rightarrow  (\gfin)_n$ and admits left adjoint 
 $
  q_n\circ  \fieldp [-] : (\gfin)_n \rightarrow \ffin. 
 $
\end{prop}

\begin{rem}
 The restriction to $(\gfin)_n$ serves  to restrict to presheaves 
taking 
 values in finite sets. (The functor $q_n \circ \fieldp[- ] : \prshv 
\rightarrow 
\f$ does  not take values in finite functors.)
 \end{rem}

\begin{cor}
\label{cor:finite_morphism_sets}
 For $X, Y \in \ob \prshv$ such that $X$ takes values in finite sets and $Y$ is 
finite, 
  $
  \hom_{\prshv} (X, Y)
$
is a finite set. 
\end{cor}

\begin{proof}
By hypothesis, there exists $n \in \nat$ such that 
the natural morphism $Y \hookrightarrow q_n \fieldp [Y]$ is injective, hence 
there is a monomorphism:
 \[
  \hom_{\prshv} (X, Y)
  \hookrightarrow 
  \hom_{\prshv} (X, q_n \fieldp [Y]) \cong 
  \hom_{\f} (q_n \fieldp [X], q_n \fieldp [Y]),
 \]
 where the second isomorphism is given by Proposition \ref{prop:q_nF_ajdunction}.

Now $q_n \fieldp [X]$ and $ q_n \fieldp [Y]$ are both finite functors of $\f$ (by 
Proposition \ref{prop:finite_polynomial}) hence  $ 
\hom_{\f} (q_n \fieldp [X], q_n \fieldp [Y])$ is a finite-dimensional 
$\fieldp$-vector space, thus a finite set.
\end{proof}

\begin{defn}
 For $X \in \ob \prshv$ that takes values in finite sets and $n \in \nat$, let 
$X_n \in \ob \prshv$ denote the image of $X \rightarrow q_n \fieldp [X]$, 
 equipped with the canonical surjection 
 $
  X \twoheadrightarrow X_n.
 $
\end{defn}

\begin{prop}
\label{prop:finite_tower}
For $X \in \ob \prshv$ that takes values in finite sets,  the natural 
surjection $X \twoheadrightarrow X_n$ is the universal map to a finite presheaf 
of degree 
$n$. 

These maps form a tower 
\[
 \xymatrix{
 &X 
 \ar[d]
 \ar[rd]
 \ar[rrd]
 \\
 \ldots
 \ar[r]
 &
 X_{n+1}
 \ar[r]
 &
 X_n 
 \ar[r]
 &
 X_{n-1}
 \ar[r]
 &
 \ldots
\  . }
\]
\end{prop}

\begin{proof}
 The key point is to check that $X_n$ is finite of degree at most $n$. This is 
clear from the commutative diagram 
 \[
  \xymatrix{
  X_n 
  \ar@{^(->}[r]
  \ar@{^(->}[d]
  &q_n\fieldp [X]
  \\
  \fieldp[X_n]
  \ar@{-->}[ru]
  \ar@{->>}[r]
  &
  q_n \fieldp [X_n], 
  \ar@{.>}[u]
  }
 \]
where the top monomorphism is given by the construction of $X_n$, the dashed 
arrow by  $\fieldp$-linear extension and the dotted arrow by the polynomial 
degree adjunction.

The universality of $X \twoheadrightarrow X_n$ follows from Proposition 
\ref{prop:q_nF_ajdunction}.
\end{proof}

\subsection{Coanalyticity of presheaves of sets}

Even under the hypothesis that $X$ takes values in finite sets, the induced 
map 
 $
  X \rightarrow \lim_\leftarrow X_n
 $ given by Proposition \ref{prop:finite_tower} 
need not be an isomorphism. 

\begin{exam}
\label{exam:Ibar}
Consider the functor $V \mapsto I_{\fieldp}(V) := {\fieldp^V}$ in $\f$, 
and take its constant-free summand $\overline{I_{\fieldp}}$. Forgetting the linear 
structure gives a connected presheaf taking values in finite sets. 

The linearization  $\field[\overline{I_{\fieldp}}]$ admits {\em no} non-trivial 
map 
to a finite 
functor. It is straightforward to reduce to proving this for $q_n \fieldp[-] 
\circ \overline{I_{\fieldp}}$, for all $n \in \nat$.
Using the identification of the subquotients of the  filtration associated to 
$\fieldp [-] \twoheadrightarrow q_n \fieldp[-]$, one shows that it is sufficient 
to show that functors of the form $\overline{I_{\fieldp}}^{\otimes i}$, for $0< i 
\in 
\nat$, admit no finite quotients. Now Kuhn's embedding theorem \cite{KI} 
implies 
that it suffices to show that $\hom_\f ( \overline{I_{\fieldp}}^{\otimes i}, S^t) 
=0$ 
 for all $t \in \nat$ and $i>0$. This follows from the case $i=1$ by using the 
exponential property of the symmetric power functors; the case $i=1$ is 
well-known, since the 
structure of $\overline{I_{\fieldp}}$ is known.
\end{exam}

\begin{prop}
\label{prop:coanalytic_finite_values}
 For $X \in \ob \prshv$ taking values in finite sets, the map 
 $$X \rightarrow \lim_{\substack{\leftarrow\\ n}} X_n$$
 is a 
bijection if and only if, for each  $V \in \ob \fdvs$, $\exists n_V$ such that 
 \[
  X (V) 
  \stackrel{\cong}{\twoheadrightarrow} 
  X_{n_V} (V).
 \]
\end{prop}

In general one must consider the category $\gps$ of presheaves on $\fdvs$ with 
values in profinite sets.

The following should be compared with the definition in \cite[Part II, page 
1078]{HLS} of an {\em analytic} functor from $\fdvs$ to $\pfset \op$, the 
opposite of the category of profinite sets.
(Henn, Lannes and Schwartz work with the opposite of the category of 
presheaves, 
whence their terminology {\em analytic} rather than {\em coanalytic} here.)

\begin{defn}
 An object $X$ of $\gps$ is coanalytic if 
 \[
  X \cong \lim_{\substack{\longleftarrow\\ i \in \cali}} X(i),
 \]
where the indexing category $\cali$ is cofiltered and small and $X(i)$ are 
finite presheaves.

Let $\gcoan \subset \gps$ denote the full subcategory of coanalytic functors. 
\end{defn}

\begin{defn}
For $X \in \ob \prshv$ (respectively $X\in \ob \gps$), let $X/ \gfin$ denote 
the 
full subcategory of the 
undercategory $X/ \prshv$ (resp. $X/ \gps$) with objects  $X \rightarrow Y$ 
with 
$Y$ finite. 
\end{defn}

\begin{lem}
 For $X \in \ob \prshv$ or $X \in \ob \gps$,
 \begin{enumerate}
  \item 
  $X/ \gfin$ has finite morphism sets;
  \item 
  given $X\rightarrow Y_i$, $i \in \{1, 2 \}$ of $X/ \gfin$, there is a diagram 
of morphisms of $X/ \gfin$:
  \[
   \xymatrix{
   & X
   \ar[ld]
   \ar[d]
   \ar[rd]
   \\
   Y_1 
   &
   Y_1 \times Y_2 
   \ar[l]
   \ar[r]
   &
   Y_2,
   }
  \]
in which the horizontal arrows are the projections; 
\item 
any object $X \rightarrow Y$ of $X/ \gfin$ is the range of a morphism in $X/ 
\gfin$:
\[
 \xymatrix{
 & X 
 \ar@{->>}[ld]\ar[rd]
 \\
 Y' 
 \ar@{^(->}[rr]
 &&
 Y 
 }
\]
where $Y' \subset Y$ is a sub-presheaf and $X \twoheadrightarrow Y'$ is 
surjective. 
\item 
for morphisms $f, g : (X \rightarrow Y_1) \rightrightarrows (X\rightarrow Y_2)$ 
of $X/ \gfin$, the  morphism 
$(X \twoheadrightarrow Y'_1) \rightarrow (X \rightarrow Y_1)$ given by the 
factorization of $X\rightarrow Y_1$ equalizes $f, g$.
 \end{enumerate}
In particular, the category $X/ \gfin$ is cofiltered. 
 
\end{lem}

\begin{proof}
 The first statement is an immediate consequence of  Corollary 
\ref{cor:finite_morphism_sets} and the second follows from the categorical 
definition of 
the product. 
 The factorization of a morphism of presheaves is clear and applies in the 
final 
statement, using the categorical property of a surjection. 
\end{proof}

\begin{defn}
Let $X^\omega \in \ob \gcoan$ denote the presheaf of profinite sets given by 
\[
 X^\omega := \lim_{\substack{\longleftarrow\\ X\rightarrow Y \in X / \gfin}} Y,
\]
equipped with the natural  (continuous) {\em coanalytic completion} map 
$
  X\rightarrow X^\omega.
$
\end{defn}

The following is clear from the definitions:

\begin{prop}
 \label{prop:coanalytic_vs_coanalytic_completion}
 A presheaf $X \in \ob \gps$ is coanalytic if and only if the natural map $X 
\rightarrow X^\omega$ is an isomorphism.
\end{prop}

\begin{exam}
\label{exam:coanalytic_comp_trivial}
 For $X \in \ob \prshv$, the category $X/ \gfin$ can be the discrete category 
with one object.   Consider the presheaf $\overline{I_{\fieldp}}$ of Example 
\ref{exam:Ibar}, where 
it was shown that there are no non-trivial maps from $\overline{I_{\fieldp}}$ to a 
finite presheaf.   It follows that $(\overline{I_{\fieldp}})^\omega = *$.
 \end{exam}

\section{The relationship with unstable algebras}
\label{sect:unstable} 

The relationship between finite presheaves and $\cala$-finite generation up to 
nilpotents is made explicit 
in this section, leading to a conceptual restatement of one of the main results 
of \cite[Part II]{HLS}.

\subsection{$\cala$-finite generation up to nilpotents}

The following is the natural extension of Definition \ref{def:A_fg} to working 
modulo nilpotents.

\begin{defn}
\label{def:A_fg_upto_nil}
An unstable algebra $K \in \ob \unstalg$ is $\cala$-finitely generated up to 
nilpotents if there exists an unstable 
algebra $L$ that is $\cala$-finitely generated and an 
 $F$-epimorphism 
 $
  \xymatrix{
  L \ar@{~>>}[r]^{F-\mathrm{epi}} &K.
 }
 $
\end{defn}

\begin{rem}
Unlike the case of $\cala$-finite generation, for $K \in \ob \unstalgc$,  the 
definition of $\cala$-finitely generated up to 
nilpotents cannot be given in terms of $QK$. 
 
For example, consider the unstable algebra $K:= U (\bigoplus_{n\geq 2} F(n))$ 
at the prime $p=2$, so that $QK \cong \bigoplus_{m\geq 1} \Sigma F (m)$; $K$ is 
 not $\cala$-finitely generated up to nilpotents. Now consider the unstable 
algebra $\fieldp \oplus  \bigoplus_{m\geq 1} \Sigma F (m)$ with trivial 
algebra structure; this has the same module of indecomposables, but is 
$F$-isomorphic to $\fieldp$, hence is $\cala$-finitely generated up to 
nilpotents.
\end{rem}

The following is an immediate consequence of Kuhn's embedding theorem 
\cite{KI} that a finite functor $F \in \f$ embeds in a finite direct sum of 
symmetric power functors.

\begin{prop}
\label{prop:conseq_emb_thm}
 A presheaf $X \in \ob \gps$ is finite if and only if there is a finite 
direct sum $\bigoplus_{i \in \mathscr{I}}S^{n_i}$ of symmetric power functors 
and a monomorphism
  $
  X \hookrightarrow \bigoplus_{i \in \mathscr{I}}S^{n_i}
 $   in $\gps$.
\end{prop}

\begin{cor}
\label{cor:equivalent_finite_fgmodnil}
 Let $K \in \ob \unstalg$ be an unstable algebra. The following 
conditions are equivalent:
 \begin{enumerate}
  \item 
  $\g K \in \ob \gps$ is finite; 
  \item
  $K$ is $\cala$-finitely generated up to nilpotents. 
 \end{enumerate}
\end{cor}

\begin{proof}
 An immediate consequence of Proposition \ref{prop:conseq_emb_thm}, Proposition 
\ref{prop:kappa_g}  and Example \ref{exam:kappa_Sn}, using the fact that 
$\kappa $ sends inclusions to $F$-epimorphisms \cite{HLS}.
\end{proof}

\begin{rem}
As observed by a referee, the key ingredient here is the fact that an unstable module $M$ is $\cala$-finitely generated modulo nilpotents if and only if the associated functor $V \mapsto (T_V M)^0$  (where $T_V$ is Lannes' $T$-functor) is  finite  (this is related to Kuhn's embedding theorem).  Using this, Corollary \ref{cor:equivalent_finite_fgmodnil} has a short direct proof, by applying the functor $\g$ to a morphism of unstable algebras of the form $UM \rightarrow K$.
\end{rem}

\subsection{Reinterpreting $\unstalg/\nil$}

The above leads to the following refinement of \cite[Theorem II.1.5]{HLS}:

\begin{thm}
\label{thm:HLS_improved}
 For $X \in \ob \gps$, the natural map $X \rightarrow X^\omega$ induces an 
isomorphism of unstable algebras 
$
  \kappa X^\omega \rightarrow \kappa X.
 $

 In particular, 
\begin{enumerate}
 \item 
 $X$ is coanalytic if and only if $X \cong \g \kappa X$;
 \item 
 $\g$ induces an equivalence of categories 
 $
 (\unstalg/ \nil)\op \stackrel{\cong}{\rightarrow} \gcoan.
$
\end{enumerate}
\end{thm}

 \section{Finitely generated presheaves are finite}

\subsection{The rank filtration}

For $X \in \ob \prshv$ and $n \in \nat$, the sections $X (\fieldp^n)$ have a 
natural right action of $\End (\fieldp^n)$, which 
restricts to a right $\Aut (\fieldp^n)$-action.

The following result is standard (and corresponds to the skeletal filtration of 
\cite[Part II]{HLS}):

\begin{prop}
\label{prop:rank_filtration}
For $X \in \ob \prshv$, there is a natural rank filtration 
\[
 X_{\leq 0}\subset X_{\leq 1} \subset \ldots \subset X_{\leq n} \subset X_{\leq 
n+1} \subset \ldots \subset X
\]
such that $X = \bigcup X_{\leq n}$, where $X_{\leq n}$  is the image of the 
 evaluation map:
\[
 X(\fieldp^n) \times_{\End(\fieldp^n)} \hom (- , \field^n) \rightarrow X.
\]
\end{prop}

\begin{defn}
 For $X \in \ob \prshv$ and $n \in \nat$, let $X_\reg (n)$ be  the set of {\em 
regular 
elements} of $X(\fieldp^n)$, namely the right $\Aut (\fieldp^n)$-set given by
 \[
  X_\reg(n) := X(\fieldp^n) \backslash X_{\leq n-1} (\fieldp^n).
 \]
\end{defn}

The following is related to the Key Lemma, \cite[Lemma II.2.1]{HLS} and its 
associated results.

\begin{lem}
\label{lem:presheaf_rank_filt}
 For $X \in \ob \prshv$ and $1\leq n \in \nat$, 
 \begin{enumerate}
  \item 
  the quotient presheaf $X_{\leq n}/ X_{\leq n-1}$ is naturally isomorphic to 
\[
 V \mapsto 
 * \amalg 
 \Big(X_\reg (n) \times_{\Aut 
(\fieldp^n)} \Surj (V, \fieldp^n)\Big)
\]
(where $\Surj (V, \fieldp^n) \subset \hom (V, \fieldp^n)$ is the set of 
surjective 
morphisms) considered as a quotient presheaf of $  X(\fieldp^n) 
\times_{\End(\fieldp^n)} \hom (- , \fieldp^n) \rightarrow X$.
\item 
 there is a natural isomorphism of $\mathrm{Aut}(V)$-sets:
\[
 X_{\leq n} (V) \backslash X_{\leq n-1} (V) \cong X_\reg (n) \times_{\Aut 
(\fieldp^n)} \Surj (V, \fieldp^n). 
\] 
 \end{enumerate}
\end{lem}

\begin{proof}
 By definition, $X_{\leq n}$ is the image 
 \[
   X(\fieldp^n) \times_{\End(\fieldp^n)} \hom (- , \fieldp^n) \twoheadrightarrow 
X_{\leq n} \hookrightarrow X
 \]
of the map induced by evaluation. From the definition of $X_\reg(n)$, this  
induces a 
surjection 
\begin{eqnarray}
 \label{eqn:quotient}
 * \amalg 
 \Big(X_\reg (n) \times_{\Aut 
(\fieldp^n)} \Surj ( -, \fieldp^n)\Big)
\twoheadrightarrow
X_{\leq n}/ X_{\leq n-1}.
\end{eqnarray}
Now $\Surj (V, \fieldp^n)$ is a free left $\Aut (\fieldp^n)$-set 
with cosets $\overline{\Surj (V, \fieldp^n)}$ in bijection with the set 
of codimension $n$ subspaces of $V$. Hence, as sets, 
$$X_\reg (n) \times_{\Aut 
(\fieldp^n)} \Surj ( V, \fieldp^n) 
\cong 
X_\reg (n) \times
\overline{\Surj (V, \fieldp^n)}. 
$$

For the first point, it remains to show that the map (\ref{eqn:quotient}) is 
injective. By construction, this is  true for sections with  $\dim V \leq n$ 
and, for $V= \fieldp^n$, 
both sides of (\ref{eqn:quotient}) identify with the set $* \amalg X_\reg (n)$. 

For the general case, given two sections $x \neq y \in X_\reg (n) 
\times_{\Aut 
(\fieldp^n)} \Surj ( V, \fieldp^n) $, using the fact that any surjection admits a 
section, there exists a morphism $\phi : \fieldp^n \hookrightarrow V$ such that 
one of the following hold: 
\begin{enumerate}
\item 
$x \phi \in X_\reg (n)$ and $y \phi = *$,  (in the case that $x$, $y$ 
correspond 
to different cosets in $\overline{\Surj (V, \fieldp^n)}$);
 \item 
$ x \phi  \neq  y \phi \in X_\reg (n)$ (in the remaining case).
\end{enumerate}
This implies the required injectivity. 

The second statement is an immediate consequence.
\end{proof}

\begin{rem}
 A morphism of presheaves $f : X \rightarrow Y$ does not in general restrict to 
a morphism $X_\reg (n) \rightarrow Y_\reg (n)$. 
\end{rem}

\begin{prop}
\label{prop:equivalent_mono}
 For a morphism of presheaves $f : X \rightarrow Y$ in $\prshv$, the following 
conditions are equivalent:
 \begin{enumerate}
  \item 
  $f$ is a monomorphism $X \hookrightarrow Y$; 
  \item 
  for each $n \in \nat$, $f$ restricts to a monomorphism $X_\reg (n) 
\hookrightarrow Y_\reg (n)$. 
 \end{enumerate}
\end{prop}

\begin{proof}
 If $f$ is a monomorphism, it is sufficient to show that, for each $n \in 
\nat$, 
$f$ sends $X_\reg (n)$ to $Y_\reg (n)$. Suppose that $x \in X_\reg (n)$; if 
$f(x)$ is not regular, then there exists a non-invertible $\alpha \in \End 
(\fieldp^n)$ such that $f(x)= f(x) \alpha$. By injectivity of $f$, it follows 
that $x = x\alpha$, a contradiction.
 
The converse is established by induction on the rank filtration. For $V \in \ob 
\fdvs$, as in the proof of  Lemma \ref{lem:presheaf_rank_filt},  
\[
X_{\leq n} (V) \backslash X_{\leq n-1} (V) \cong 
X_\reg (n) \times
\overline{\Surj (V, \fieldp^n)}.
\]
Hence, the inclusion $X_\reg (n) \hookrightarrow Y_\reg (n)$ of right $\Aut 
(\fieldp^n)$-sets induces a monomorphism
\[
  X_{\leq n} (V) \backslash X_{\leq n-1} (V)
   \hookrightarrow 
   Y_{\leq n} (V) \backslash Y_{\leq n-1} (V)
\]
and thus, by induction upon $n$, $X_{\leq n } (V) \hookrightarrow Y_{\leq n} 
(V)$. The result 
follows by passage to the colimit as $n \rightarrow \infty$.
 \end{proof}

\begin{cor} 
\label{cor:criterion_mono}
 For a morphism of presheaves $f : X \rightarrow Y$ in $\prshv$ such that $X = 
X_{\leq n}$, $f$ is a monomorphism if and only if 
 \[
  f : X(\fieldp^n) \hookrightarrow Y(\fieldp^n)
 \]
is a monomorphism of sets. 
\end{cor}

\begin{proof}
 By hypothesis, $X_\reg (k) = \emptyset$ for $k > n$, hence the result follows 
from Proposition \ref{prop:equivalent_mono}.
\end{proof}

 \subsection{Finite generation implies finite}

 \begin{nota}
  For $Z$ a finite right $\End (\fieldp^n)$-set, denote by 
\begin{enumerate}
 \item 
 $X_Z \in \ob \prshv$ the induced presheaf $V \mapsto Z \times_{\End (\fieldp^n) 
} 
\hom (V, \fieldp^n)$
 \item 
 $G_Z \in \ob \f$ the induced functor $V \mapsto \fieldp[Z ] \otimes _{\End 
(\fieldp^n) } \fieldp [\hom (V, \fieldp^n)]$
\end{enumerate}
equipped with the morphism $
  X_Z \rightarrow G_Z
 $ of $\prshv$ 
 induced by the canonical inclusion of right $\End (\fieldp^n)$-sets $Z 
\hookrightarrow \fieldp [Z]$.
 \end{nota}

\begin{thm}
\label{thm:fg_implies_finite}
   For $Z$ a finite right $\End (\fieldp^n)$-set, there exists $ t \in \nat$ 
such 
that the composite
   \[
      X_Z \rightarrow G_Z \twoheadrightarrow q_t G_Z
   \]
is a monomorphism. In particular, $X_Z$ is finite of degree at most $t$. 
\end{thm}

 \begin{proof}
  By Corollary \ref{cor:criterion_mono}, it suffices to exhibit $t \in \nat$ 
such that the natural surjection $G_Z \twoheadrightarrow q_t G_Z$ is a 
bijection 
when evaluated on $\fieldp^n$. 
  The functor $G_Z$ is a quotient of a finite direct sum of copies of the 
functor $\fieldp[\hom (-, \fieldp^n)]$.
  The latter takes finite-dimensional values and is dual to a locally finite 
functor, namely is the inverse limit of its finite quotients. 
   The same therefore holds for $G_Z$, which  implies the existence of such a 
$t$. 
  
  The final statement follows, since $q_t G_Z$ is a finite functor of 
polynomial 
degree at most $t$. 
 \end{proof}

Recall the theory of unstable algebras of finite transcendence degree from 
Section \ref{subsect:HLS_trans_deg}, 
in particular Theorem \ref{thm:HLS_trans_deg}.

 \begin{cor}
 \label{cor:trans_deg_finite_presheaf}
  Let $K$ be an unstable algebra of finite transcendence degree. Then $K$ is 
$\cala$-finitely generated up to nilpotents if and only if $\g K$ takes values 
in finite sets. 
 \end{cor}

 \begin{proof}
 If $K$ is $\cala$-finitely generated up to nilpotents, then $\g K$ is finite, 
by Corollary 
\ref{cor:equivalent_finite_fgmodnil}, hence takes values in finite sets.
 
Conversely, if $K$ has finite transcendence degree $d$ and $\g K$ takes values 
in finite sets, then $\g K \cong X_{\g K (\fieldp^d)}$, where $\g K (\fieldp^d)$ 
is a finite right $\End (\fieldp^d)$-set.  The presheaf $X_{\g K (\fieldp^d)}$ is 
finite  by Theorem \ref{thm:fg_implies_finite}, hence $K$ is $\cala$-finitely 
generated up to nilpotents, by Corollary \ref{cor:equivalent_finite_fgmodnil}. 
 \end{proof}

\begin{rem}
Corollary \ref{cor:trans_deg_finite_presheaf} applies  when $\g K 
(\fieldp^d)$ is a Noetherian right $\End (\fieldp^d)$-set (see Theorem 
\ref{thm:HLS_noetherian}). However, the Corollary shows that, for 
$\cala$-finite 
generation up to nilpotents (as opposed  to Noetherian  up to nilpotents),  only 
the much weaker condition of finite values is required. 
\end{rem}
 
\subsection{Examples}
 
The prime $p$ is taken to be $2$ in this section.
 
 \begin{exam}
 \label{exam:grass}
  For $n \in \nat$, let $\grass_{\leq n} \in \ob \prshv$ denote the presheaf 
defined by  $\grass_{\leq n}(V):= \gl_n \backslash \hom  (V, \fieldp^n)$. 
  A linear embedding $\fieldp^{n-1} \hookrightarrow \fieldp^{n}$ induces an 
inclusion (independent of the choice)
   $
   \grass_{\leq n-1} \hookrightarrow \grass_{\leq n}.
  $ 
Let $\grass_n$ denote the presheaf defined by the pushout diagram in $\prshv$
\[
 \xymatrix{
 \grass_{\leq n-1}
 \ar@{^(->}[r]
 \ar[d]
 \po
 &
 \grass_{\leq n}
 \ar[d]
 \\
 {*} 
 \ar[r]
 &
 \grass_n.
 }
\]
Hence $\grass_n$ identifies as the presheaf $* \amalg \overline{\mathrm{Surj} ( 
- ,\fieldp^n)}$, using the notation 
of the proof of Lemma \ref{lem:presheaf_rank_filt}.

Each of these presheaves is connected and takes values in finite sets. 
Moreover, 
$\grass_{\leq n}$ is induced by the right $\End(\fieldp^n)$-set $\gl_n 
\backslash 
\hom (\fieldp^n, \fieldp^n)$ (which is Noetherian),  
whereas $\grass_n$ is induced (for $n>0$) by the pointed $\End 
(\fieldp^n)$-set $*_0 \amalg *_n$, where $*_n$ is regular (at the prime $2$, 
this 
can be identified with $\Lambda^n (\fieldp^n)$). The latter is {\em 
not} Noetherian for $n>1$. 
  
The unstable algebra $\kappa (\grass_{\leq n})$ is the 
Dickson algebra $D(n) := H^* (B\fieldp^n ; \fieldp) ^{\gl_n}$ whereas 
$\kappa (\grass_{n})$ 
is $\fieldp  \oplus \omega_n D(n)$, where $\omega_n$ denotes the top Dickson 
invariant. For $n=1$ these coincide, whereas for $n>1$ they differ: indeed 
$D(n)$ is a Noetherian algebra whereas, for $n >1$, $\fieldp  \oplus \omega_n 
D(n)$ is easily seen not to be Noetherian or even $\cala$-finitely generated. 
However, Theorem \ref{thm:fg_implies_finite} implies that $\fieldp  \oplus 
\omega_n D(n)$ is $\cala$-finitely generated up to nilpotents. 
\end{exam}

 \begin{exam}
 \label{exam:grass_2}
  Consider the finite presheaf $\grass_2$ over $\fieldt$. It is a basic 
calculation  that $q_2 \fieldt [\grass_2] \cong \Lambda^2 \oplus \fieldt$, 
so that 
there is a canonical inclusion 
  \[
   \grass_2 \hookrightarrow \Lambda^2 \oplus \fieldt
  \]
in $\prshv$. Since $\grass_2$ is connected, it follows that there is  
an inclusion $\grass_2 \hookrightarrow \Lambda^2$ (as a presheaf of sets, 
$\fieldt \oplus \Lambda^2$ identifies as $\Lambda^2 \amalg \Lambda^2$). This 
gives one approach to showing that $\End _{\prshv} (\grass_2)$ is the monoid 
underlying $\fieldt$. 

The inclusion $\grass_2 \hookrightarrow \Lambda^2$ corresponds to the canonical 
$F$-epimorphism in $\unstalg$:
\[
 U (\Lambda^2 (F(1)))
 \rightarrow 
 \fieldt
 \oplus 
 \omega_2 D(2).
\]
 \end{exam}

 \begin{exam}
The following example complements Example \ref{exam:coanalytic_comp_trivial}.
Consider the presheaves $\grass_{\leq n}$ of Example \ref{exam:grass} and 
form 
the colimit
  \[
   \grass_{\leq \infty} := \bigcup _n \grass_{\leq n} 
  \]
in $\prshv$. This exhibits $\grass_{\leq \infty}$ as a colimit of connected 
presheaves induced from Noetherian $\End (\fieldt^n)$-sets, whereas 
$\grass_{\leq 
\infty}$ is not finitely generated, although it does take values in finite 
sets. 
The fact that the augmentation ideal $\overline{D(n)}\subset D(n)$ is zero in 
degrees $< 2^{n-1}$ implies  that  $(\grass_{\leq 
\infty})^\omega = *$.
 \end{exam}

\section{The growth function $\gamma$}
 
 The growth functions introduced in this section provide a useful first 
approximation 
 to the property of finiteness for presheaves.

\subsection{The function $\gamma_X$}

 \begin{defn}
For  $\emptyset \neq X  \in \ob \prshv$ that takes values in finite sets, let 
$\gamma_X : \nat \rightarrow \reals_{\geq 0}$ 
be the growth function defined by 
\[
 t \mapsto \gamma_X (t):= \log_p |X (\fieldp^t) |.
\]
 \end{defn}

 \begin{rem}
  If $F \in \ob \f$ takes finite-dimensional values, then 
$
 \gamma_F (t) = \dim F (\fieldp^t),
$
so that $\gamma_F$ coincides with the growth function 
considered in \cite{KI}.
 \end{rem}

 The following characterization of polynomial degree in terms of the growth 
function is useful:
 
 \begin{prop}\cite{KI}
  \label{prop:poly_degree_growth}
 Let $F \in \ob \f$ be a functor that takes finite-dimensional values.  
 The following are equivalent:
 \begin{enumerate}
  \item 
  $F$ is polynomial of degree $d$; 
  \item 
  $\gamma_F$ is a polynomial function of degree $d$.
 \end{enumerate}
 \end{prop}

\begin{prop}
\label{prop:finite_growth_poly}
 Suppose that $\emptyset \neq X  \in \ob \prshv$ is finite of degree $d$. Then 
 \begin{enumerate}
  \item 
  $|X (V) | < \infty$ for every $V \in \ob \fdvs$; 
  \item 
$\gamma_X (t)= O (t^d)$.
 \end{enumerate}
\end{prop}

\begin{proof}
By hypothesis, $\emptyset \neq X \hookrightarrow F_X$, where $F_X \in \ob \f$ 
is 
a finite functor of degree $d$. 
Hence $\gamma_X \leq \gamma_{F_X}$; since $F_X$ is of degree $d$, 
$\gamma_{F_X}$ 
is a polynomial function of degree $d$, by Proposition \ref{prop:poly_degree_growth}.
\end{proof}

\subsection{Applications}

Such growth functions play an important role in the work of Grodal \cite{G} and 
Lannes and 
Schwartz \cite{LS}. 

\begin{nota}
 \cite{G}
 For functions $f, g : \nat \rightarrow \reals$, write $f \lesssim g$ if, for 
all $\epsilon >0$, there exists $N \in \nat$ such that 
 $f(t) \leq (1+\epsilon ) g(t)$ for all $t \geq N$. 
\end{nota}

\begin{thm}
\label{thm:grodal}
 \cite[Theorem 3.3]{G}
Let $E$ be a connected, nilpotent finite Postnikov system with finite $\pi_1 E$ 
and finitely-generated homotopy groups. 
Then there exist $0< c,C \in \nat$ such that, as functions of $t$:
\[
 c t^d \lesssim \log_p  | \gtop E (\fieldp^t)| \lesssim Ct^d, 
\]
where for  $k := \sup \{ i | \pi_i( E) _{p} \neq 0 \}$, 
$d = k$  if $\pi_k E$ has $p$-torsion, otherwise $d=k-1$.   
\end{thm}

\begin{rem}
 This should be compared with the argument used in the proof of \cite[Theorem 
0.1]{LS}, using the fact (see Theorem \ref{thm:Lannes_g_gtop}) that, under 
suitable hypotheses upon the topological space $E$,  
  cohomology induces an isomorphism of presheaves 
   $
   \gtop E \stackrel{\cong}{\rightarrow} \g H^* E.
  $
In this situation, Grodal's theorem implies  that 
$\gamma _{\g H^* E}(t) = O (t^d)$.

However, the arguments of \cite{G,LS} only  provide bounds on the growth 
function; in particular, they do not show that $\g 
H^*E$ is a finite presheaf (see Example \ref{exam:split_rank} below).
\end{rem}

We do, however, have the following:

\begin{cor}
 Let $E$ satisfy the hypotheses of Theorem \ref{thm:grodal} and suppose that 
$\gtop E$ is a finite presheaf. Then $\gtop E$ has degree at least $d$ (for 
$d$ as in the Theorem).
\end{cor}

\begin{exam}
 \label{exam:split_rank}
The property that $\gamma_X$ is a polynomial function does not imply 
that $X$ is a finite presheaf.

Consider any finite, constant-free functor $0 \neq F \in \f$ of polynomial 
degree $d \geq 2$, so that $\gamma _F (t) =O (t^d)$. 
To be concrete, we take $F= \Lambda^2$ over the field $\fieldt$.

The rank filtration of Proposition \ref{prop:rank_filtration} provides the 
decomposition of the underlying set
\[
 F (V) = \coprod_n F_{\leq n} (V) \backslash F_{\leq n-1} (V)  
\]
and
\[
 F_{\leq n} (V) \backslash F_{\leq n-1} (V) \cong 
 F_\reg (n) \times_{\Aut 
(\fieldt^n)} \Surj (V, \fieldt^n) 
\]
by Lemma \ref{lem:presheaf_rank_filt} and, for $n \geq 1$, the quotient $F_n 
(V) 
 := F_{\leq n} (V) / F_{\leq n-1} (V)$ in presheaves identifies with 
\[
 V \mapsto 
 * \amalg 
 \Big(F_\reg (n) \times_{\Aut 
(\fieldt^n)} \Surj (V, \fieldt^n)\Big).
\]

Consider the connected presheaf 
 $
\tilde{F}:=  \bigvee _{n\geq 1} F_n. 
 $
 For $V \in \ob \fdvs$, by construction,  $\tilde{F}(V)$ has the same 
underlying 
finite set as $F(V)$, but a very different $\End (V)$-structure.

To show that the presheaf $ \tilde{F}$ is not finite, it suffices to show that, 
for $G \in \ob \f$ a finite, constant-free functor, there is no non-trivial map 
from 
$F_n$ to $G$ for $n \gg 0$. 

Now $\hom _{\prshv} (F_n, G) \cong \hom_\f(\fieldt [F_n], G)$ and $\fieldt[F_n]$ 
splits as $\fieldt \oplus \overline{\fieldt [F_n]}$ where $\overline{\fieldt[F_n]} 
(V)=0$ if 
$\dim V < n$, by construction of $F_n$. Since $G$ is constant-free, $\hom_\f 
(\fieldt [F_n], G) \cong \hom_\f (\overline{\fieldt[F_n]}, G)$ and, since $G$ is 
finite, it follows by connectivity  arguments that 
$\hom_\f (\overline{\fieldt[F_n]}, G)=0$ for $n \gg 0$.
\end{exam}

\part{Unstable Hopf algebras and presheaves of $p$-groups} 
\section{Hopf algebras in $\unstalg$}
\label{sect:hopf}

\subsection{Preliminaries}

As usual, $\unstalg$ denotes the category of unstable algebras over the 
mod $p$ Steenrod algebra $\cala$. 

\begin{defn}
 Let $\hopfk$ be the category of cogroup objects in $\unstalg$. 
 
 \begin{enumerate}
  \item 
   An object of $\hopfk$ is a commutative $\fieldp$-Hopf algebra $H$, 
 such that the underlying algebra is an unstable algebra over the Steenrod 
algebra, and the structure morphisms $\Delta : H \rightarrow H \otimes H$ (the 
diagonal or coproduct) and 
 $\chi : H \rightarrow H$ (the conjugation or antipode) are morphisms of modules 
over the 
Steenrod 
algebra.
 \item 
 A morphism $H_1 \rightarrow H_2$ of $\hopfk$ is a morphism of $\fieldp$-Hopf 
algebras that is $\cala$-linear.
 \end{enumerate}
Let $\hopfkc \subset \hopfk$ denote the full subcategory of connected objects 
($H$ such that $H^0 = \fieldp$).
\end{defn}

\begin{rem}
It is not assumed that the coproduct $\Delta$ is cocommutative.
\end{rem}

Here we focus upon connected unstable Hopf algebras. his is not a serious restriction, since  the general case can be  treated by using the following:

\begin{lem}
\label{lem:hopf_K_dim0}
 Let $H \in \ob \hopfk$ be a Hopf algebra concentrated in degree zero with 
$\dim 
H^0 < \infty$. 
 Then 
 $
  H \cong \fieldp ^{\spec H^0},   
 $ 
where $\spec H^0$ is a finite group.
\end{lem}

The following is a key fact:

\begin{prop}
\label{prop:hopf_A_fg}
 Let $H \in \ob \hopfkc$ and $K \subset H$ be a sub unstable algebra such that 
$K$ 
is $\cala$-finitely generated. Then there exists $H_K \subset H$ in $\hopfkc$ 
such that 
 \begin{enumerate}
  \item 
  $H_K$ is $\cala$-finitely generated as an unstable algebra;
  \item 
  $K \subset H_K$ as unstable algebras.
 \end{enumerate}
\end{prop}

\begin{proof}

By hypothesis, there exists a finite graded vector subspace $V_K \subset 
\overline{K}$ such that the induced morphism of unstable algebras 
  $
  U (F (V_K)) 
  \rightarrow 
  K
 $ 
is surjective, where $F (V_K)$ is the free unstable module on $V_K$. 

The vector space $V_K$ is contained within a finite-dimensional sub 
$\fieldp$-coalgebra $C_K \subset H$. (This is a standard fact; in this graded 
connected setting, the proof is straightforward.) Let $H_K$ denote the 
sub unstable algebra of $H$ generated by $C_K$; by construction, this 
contains $K$. Moreover, since the coproduct is $\cala$-linear,  it is 
straightforward to check that $H_K$ is stable under the coproduct. In the 
connected setting, stability under the conjugation is automatic \cite{MM}, 
hence $H_K$ is an object of $\hopfkc$, as required. 
\end{proof}

Recall  that cokernels exist in $\hopfkc$ (see \cite{MS1,MS2} for example). In 
particular, 
if $K \hookrightarrow H$ is a monomorphism of $\hopfkc$, 
then the cokernel $H \twoheadrightarrow H \hq K$ has underlying unstable 
algebra 
given by the pushout in $\unstalg$:
\[
 \xymatrix{
 K
 \ar[r]
 \ar[d]
 \po
 &
 H 
 \ar[d]
 \\
 \fieldp 
 \ar[r]
 &
 H \otimes_K \fieldp.
 }
\]

\begin{lem}
\label{lem:ses_Hopf_group}
Let $H \in \ob \hopfkc$ such that $\g 
H$ takes values in finite sets. 
Then the Hopf algebra structure of $H$ induces a natural group structure on $\g 
H$, so that 
$\g H$ takes values in finite groups. 

Let $H' \subset H$ be a sub Hopf algebra in $\hopfkc$ and consider the 
associated 
sequence $H' \hookrightarrow H \twoheadrightarrow H\hq H'$. This induces 
a short exact sequence of finite groups:
 \[
  \g (H\hq  H') \stackrel{\lhd}{\hookrightarrow} \g H \twoheadrightarrow \g H'
 \]
 \end{lem}

\begin{proof}
 The first statement is a formal consequence of the natural isomorphism $\g (K 
\otimes L) \cong \g K \times \g L$, for $K, L \in \unstalg$ and the fact that 
$H$ is a cogroup object in $\unstalg$. 
 
The underlying sequence of 
pointed sets
\[
  \g (H\hq  H') \stackrel{}{\hookrightarrow} \g H \twoheadrightarrow \g H'
 \]
is  exact as pointed sets, since the underlying unstable algebra of $H\hq H'$ 
is 
$H \otimes_{H'} \fieldp$. The morphisms are respectively injective and 
surjective as morphisms of pointed sets, by \cite[Corollary II.1.4]{HLS}. In 
particular, both $\g H'$ and $\g 
(H\hq H')$ take values in finite sets.  

The naturality of the group structure implies that these are group morphisms; 
the result follows.
\end{proof}

\subsection{The primitive filtration}

Recall that, if $H$ is a connected Hopf algebra with augmentation ideal 
$\overline{H}$, then the module of  primitives $PH$ is 
 the kernel of the reduced diagonal 
\[
 \overline{H}
 \stackrel{\overline{\Delta}}{\rightarrow }
  \overline{H} \otimes  \overline{H}.
\]

\begin{lem}
\label{lem:UPH_H}
 For $H \in \ob \hopfkc$, the module of primitives $PH \subset H$ is a 
non-trivial sub unstable module and the canonical morphism 
 \[
  U PH \rightarrow H,
 \]
where $U : \unst \rightarrow \unstalg $ is the  enveloping 
algebra functor, is a morphism of $\hopfkc$ when $U PH$ is equipped with the 
primitively-generated Hopf algebra structure. 

Moreover, this morphism is injective; $UPH$ is the largest 
primitively-generated sub Hopf algebra of $H$.
\end{lem}

\begin{proof}
This is standard. The injectivity statement follows from the natural 
identification $PH \cong P U PH$ together with the fact that a non-zero element 
of minimal degree in the kernel is necessarily primitive.
\end{proof}

As usual, one has the primitive filtration of an object of $\hopfkc$ (cf. 
\cite{MS1}, for example).

\begin{defn}
For $H \in \ob \hopfkc$, recursively define the natural sequence of quotients 
in $\hopfkc$:
\[
 \xymatrix{
 H 
 \ar@{=}[d]
 \ar[rd] 
 \ar[rrd]
 \ar@{.>}[rrrd]
 \\
 H_0 
 \ar@{->>}[r]
 &
 H_1 
 \ar@{->>}[r]
 &
 H_2 
 \ar@{.>>}[r]
 &
 \ldots 
 }
\]
where, for each $n\in \nat$, $H_{n+1}$ is the cokernel in $\hopfkc$ of $U 
PH_{n} \hookrightarrow H_{n}$.
\end{defn}

\begin{lem}
\label{lem:prim_filt}
 For $H \in\ob \hopfkc$, 
 \begin{enumerate}
  \item 
  $lim_\rightarrow H_n = \fieldp$; 
  \item 
  if $H$ is $\cala$-finitely generated (as an unstable algebra), $H_n = \fieldp$ 
for $n \gg 0$.
 \end{enumerate}
\end{lem}

\begin{proof}
 Straightforward.
\end{proof}

The nilpotent filtration behaves well when working modulo nilpotents, due to 
the following Lemma: 

\begin{lem}
\label{lem:prim_mod_nil}
 Let $H \rightarrow H'$ be a morphism of $\hopfkc$ such that the underlying 
morphism of unstable algebras is an $F$-isomorphism (equivalently induces an 
isomorphism in $\unst/ \nil$). Then 
 the induced map $PH \rightarrow PH'$ is an isomorphism in $\unst / \nil$. 
\end{lem}

\begin{proof}
 By definition, $PH \rightarrow PH'$ fits into a commutative diagram in $\unst$:
 \[
  \xymatrix{
  0 
  \ar[r]
  &
  PH 
  \ar[d]
  \ar[r]
  &
  \overline{H} 
  \ar[d]
  \ar[r]^{\overline{\Delta}}
  &
    \overline{H} \otimes \overline{H}
    \ar[d]
    \\
    0 
  \ar[r]
  &
  PH' 
  \ar[r]
  &
  \overline{H'} 
  \ar[r]_{\overline{\Delta}}
  &
    \overline{H'} \otimes \overline{H'}
    \\
  }
 \]
in which the rows are exact. After passage to $\unst/ \nil$, the two right hand 
vertical arrows are isomorphisms, by the hypothesis, hence so is the left hand 
one. 
\end{proof}

\begin{prop}
\label{prop:hopf_F_iso_primitive_filt}
 Let $H \rightarrow H'$ be a morphism of $\hopfkc$ such that the underlying 
morphism of unstable algebras is an $F$-isomorphism and $\g H$ takes values in 
finite groups.  Then the induced morphisms $
  H_n \rightarrow H_n'
 $ of the primitive filtration 
 are $F$-isomorphisms.
\end{prop}

\begin{proof}
 By induction, it suffices to prove the case $n=1$. The morphism $H \rightarrow 
H'$ induces a commutative diagram of sequences  in $\hopfkc$:
 \[
  \xymatrix{
  UP H 
  \ar@{^(->}[r]
  \ar[d]
&
H 
\ar[d]
\ar@{->>}[r]
&
H_1 
\ar[d]
\\
  UP H' 
  \ar@{^(->}[r]
&
H' 
\ar@{->>}[r]
&
H'_1. 
  }
 \]
By hypothesis, $H \rightarrow H'$ is an $F$-isomorphism; Lemma 
\ref{lem:prim_mod_nil} implies that 
$PH \rightarrow PH'$ is an isomorphism in $\unst/ \nil$ and this implies that 
$U PH \rightarrow UPH'$ is an $F$-isomorphism (by \cite{K_comp}). 

Applying Lemma \ref{lem:ses_Hopf_group} gives a morphism between short exact 
sequences of functors to finite 
groups. It follows that $\g H'_1 \rightarrow \g H_1$ is an isomorphism, hence 
that  $H_1 
\rightarrow H'_1$ 
 is an $F$-isomorphism.
\end{proof}

\begin{prop}
\label{prop:gH_values_finite_p-groups}
 Let $H \in \ob \hopfkc$ have underlying unstable algebra that is 
$\cala$-finitely generated up to  nilpotents. 
 Then $\g H $ takes values naturally in finite $p$-groups. 
\end{prop}

\begin{proof}
 By hypothesis, there exists a sub unstable algebra $K \subset H$ such that $K$ 
is $\cala$-finitely generated and the inclusion is an $F$-epimorphism. 
 Proposition \ref{prop:hopf_A_fg} provides $K \subset H_K \subset H$ with $H_K 
\in \ob \hopfkc$ that is also $\cala$-finitely generated. Clearly $H_K \subset 
H$ is also an $F$-epimorphism, hence we may suppose without loss of generality 
that $H$ is $\cala$-finitely generated as an unstable algebra, so that the 
primitive filtration is finite. (Alternatively, Proposition 
\ref{prop:hopf_F_iso_primitive_filt} can be used.)
 
 Dévissage using the primitive filtration together with Lemma 
\ref{lem:ses_Hopf_group} allows reduction to the case where $H$ is primitively 
generated. If $H \cong U PH$, then 
 $\g H$ is given by the underlying set-valued functor of $V \mapsto \hom _\unst 
(PH, H^* (BV))$. It is straightforward to check that  the group structure of $\g 
H$ corresponds to the elementary 
abelian $p$-group structure of $\hom _\unst (PH, H^* (BV))$; 
 in particular, $\g H$ takes values in finite $p$-groups. 
\end{proof}


\section{Functors to finite $p$-groups}
\label{sect:fun_fpg}

Motivated by Proposition \ref{prop:gH_values_finite_p-groups}, this section 
studies 
presheaves of finite $p$-groups.

\subsection{$p$-finiteness}

\begin{defn}
 Let $\fpg$ denote the category of functors from $\fdvs \op$ to the category of 
finite $p$-groups (i.e. presheaves of finite $p$-groups).
\end{defn}

\begin{rem}
 The category $\f$ is a full subcategory of $\fpg$. Forgetting the group 
structure gives a faithful functor 
 $
  \fpg \rightarrow \prshv.
 $
\end{rem}

Recall that the Frattini subgroup $\Phi (G) \subset G$ of a finite group $G$ is 
the intersection of all maximal proper subgroups of $G$. 
If $G$ is a finite $p$-group then
\[
 \Phi G = [G, G ]G^p
\]
and $\Phi G$ is the minimal normal subgroup of $G$ such that the quotient $G/ 
\Phi G$ is $p$-elementary abelian.

The above description makes it clear that a morphism of finite $p$-groups, $G_1 
\rightarrow G_2$ restricts to 
$\Phi G_1 \rightarrow \Phi G_2$ and thus induces a morphism of 
$\fieldp$-vector 
spaces 
\[
 G/ \Phi G_1 
 \rightarrow 
 G/ \Phi G_2.
\]

\begin{defn}
 For $G$ a finite $p$-group, let $\Phi_n G$ denote the $p$-derived series (or 
Frattini series) of $G$, defined 
 recursively by $\Phi_0 G = G$ and $\Phi_{n+1} G = \Phi (\Phi_n G)$. 
 
 For a finite $p$-group, this series is finite (i.e. $\Phi_N G=\{e \}$ for 
$N\gg 
0$). 
\end{defn}

\begin{lem}
\label{lem:naturality_Frattini_series}
 Let $f: G_1 \rightarrow G_2$ be a morphism of finite $p$-groups, then for 
$n \in \nat$, $f$ 
restricts to a morphism 
 \[
  \Phi_n f  : \Phi_n G_1 \rightarrow \Phi_n G_2, 
 \]
thus induces a natural morphism of graded $\fieldp$-vector spaces:  
\[
\bigoplus _{n\geq 0} \Phi_n f  : 
 \bigoplus _{n\geq 0} \Phi_n  G_1/\Phi_{n+1} G_1 
 \rightarrow 
 \bigoplus _{n\geq 0} \Phi_n G_2/ \Phi_{n+1} G_2. 
\]
\end{lem}

\begin{proof}
 A straightforward induction upon $n$. 
\end{proof}

\begin{prop}
 Let $\mathfrak{G} \in \ob \fpg$ be a presheaf  of finite 
$p$-groups. The  $p$-derived series induces a natural series 
 \[
 \ldots \subset \Phi_{n+1} \mathfrak{G} \subset \Phi_{n} \mathfrak{G} \subset 
\ldots 
\subset \mathfrak{G}
 \]
such that, evaluated on $V \in \ob \fdvs$, $(\Phi_n \mathfrak{G})(V) = \Phi_n 
(\mathfrak{G} (V))$.

The associated graded 
\[
 \bigoplus _{n \geq 0} \Phi_{n} \mathfrak{G}/  \Phi_{n+1} \mathfrak{G}
\]
is an $\nat$-graded functor with values in $\fdvs$. For any $V \in \ob \fdvs$, 
$\Phi_{i} 
\mathfrak{G}/  \Phi_{i+1} \mathfrak{G} (V) = 0 $ for $i \gg 0$. 
\end{prop}

\begin{proof}
 An immediate consequence of the naturality of the $p$-derived series, 
established in Lemma \ref{lem:naturality_Frattini_series}. 
\end{proof}

\begin{defn}
 A functor $\mathfrak{G} \in \ob \fpg$  to the category of finite $p$-groups is 
$p$-finite if $ \bigoplus _{n \geq 0} \Phi_{n} \mathfrak{G}/  \Phi_{n+1} 
\mathfrak{G}$ is 
a finite functor, considered as an object of the category $\f$ (forgetting the 
grading). 
\end{defn}

The notion of a polynomial functor\footnote{The author is grateful to Christine 
Vespa for pointing out that this notion of polynomial provides an 
equivalent condition.}  to the category of groups ({\em à la} Baues-Pirashvili 
\cite{BP})  is recalled in Section \ref{sect:poly}.

\begin{thm}
\label{thm:p-finite_equiv}
 For $\mathfrak{G} \in \ob \fpg$, the following conditions are equivalent: 
 \begin{enumerate}
  \item 
  $\mathfrak{G}$ is $p$-finite;
  \item 
  $\mathfrak{G}$ is polynomial;
  \item 
  Both the following conditions are satisfied:
   \begin{enumerate}
  \item 
  $\Phi_N \mathfrak{G} =0$ for $N \gg 0$ (uniformly); 
  \item 
  each $\Phi_{i} \mathfrak{G}/  \Phi_{i+1} \mathfrak{G}$ is a finite functor of 
$\f$;
 \end{enumerate}
 \item 
 $\mathfrak{G}$ has a composition series;
  \item 
  the growth function $\gamma_\mathfrak{G}$ satisfies $\gamma_\mathfrak{G}(t) = 
O (t^d)$ for some $d \in \nat$. 
 \end{enumerate}
If $\mathfrak{G}$ takes values in finite abelian $p$-groups, this is equivalent 
to 
\begin{itemize}
 \item 
 $\mathfrak{G}$ is finite as a functor to the abelian category of finite 
abelian 
$p$-groups.
\end{itemize} 
\end{thm}

\begin{proof}
 The result is proved by reducing to the case where $\mathscr{G}$ takes values 
in $\fieldp$-vector spaces, where the result is standard, 
 for instance by applying Proposition \ref{prop:finite_polynomial} and 
Proposition \ref{prop:poly_degree_growth}.
 
The equivalence of $p$-finiteness and polynomiality is proved by using the 
thickness of the polynomial property established in Corollary 
\ref{cor:thick_poly}.
\end{proof}

The relevance of $p$-finiteness is shown by the following result:

\begin{cor}
\label{cor:finite_generation_implies_finite} 
 Let $\mathfrak{G}\in \ob \fpg$ be a functor to finite $p$-groups such that  
the underlying presheaf $\mathfrak{G}\in \ob \prshv$ is finite.  Then 
$\mathfrak{G}$ is $p$-finite. 
\end{cor}

\begin{proof}
The hypothesis that the underlying presheaf is finite implies that 
$\gamma_\mathfrak{G}= O (t^d)$ for some $d \in \nat$, by Proposition 
\ref{prop:finite_growth_poly}, hence the result follows from Theorem 
\ref{thm:p-finite_equiv}.
\end{proof}

\begin{exam}
Corollary \ref{cor:finite_generation_implies_finite} applies to the case 
$\mathfrak{G} = \g H$, where $H \in \ob \hopfkc$ is $\cala$-finitely generated 
up to  nilpotents. 
\end{exam}

\subsection{Coanalyticity of $p$-finite functors to $p$-groups}

For any $\mathfrak{G}\in \ob \fpg$, composition with the group  ring functor 
$\fieldp[- ]$ (and forgetting the ring structure) gives a functor 
$\fieldp[\mathfrak{G}] \in \f$ that takes finite-dimensional values. Example 
\ref{exam:Ibar} shows that such a functor need not be {\em coanalytic} (the 
inverse limit of its finite quotients). 

This issue is resolved when one imposes  $p$-finiteness: 

\begin{prop}
\label{prop:gp_ring_coanalytic}
 Let $\mathfrak{G}\in \ob \fpg$ be $p$-finite. Then $\fieldp [\mathfrak{G}] \in 
\ob \f$ is coanalytic.
\end{prop}

\begin{proof}
 Let $I \mathfrak{G} \subset \fieldp [\mathfrak{G}]$ denote the augmentation 
ideal (kernel in $\f$ of the augmentation $\fieldp [\mathfrak{G}] \rightarrow 
\fieldp$). The powers of the augmentation ideal 
 induces a decreasing filtration 
 \[
  \ldots \subset I^{k+1} \mathfrak{G} \subset  I^{k+1} \mathfrak{G} \subset 
\ldots \subset I \mathfrak{G} \subset \fieldp [\mathfrak{G}]
 \]
and hence an inverse system of quotients 
$
 \fieldp[\mathfrak{G}] \twoheadrightarrow \fieldp[\mathfrak{G}] / I^k 
\mathfrak{G}.
$

To establish the result, it suffices to show that 
\begin{itemize}
 \item 
 each quotient $\fieldp[\mathfrak{G}] / I^k \mathfrak{G}$ is finite; 
 \item 
 $\fieldp [\mathfrak{G}] \cong \lim_{\leftarrow} \fieldp[\mathfrak{G}] / I^k 
\mathfrak{G}$.
\end{itemize}

This is proved using standard results on this filtration (see \cite{Passi}, for 
example).

For the first statement,  since $I^k \mathfrak{G}/ I^{k+1} \mathfrak{G}$ is a 
quotient of $(I \mathfrak{G} / I^2 \mathfrak{G})^{\otimes k}$,
it suffices to show that $I \mathfrak{G} / I^2 \mathfrak{G}$ is a finite 
functor. But the latter is equivalent to the functor $\mathfrak{G}_{ab} \otimes 
\fieldp \cong \mathfrak{G}/ \Phi \mathfrak{G}$. 
The $p$-finiteness hypothesis implies that this is a finite functor of $\f$: it 
takes 
finite-dimensional values and has polynomial growth (by Theorem 
\ref{thm:p-finite_equiv}), hence is finite. 

For the second statement, it suffices to show that, for any $V\in \ob \fdvs$, 
there exists $k_V \in \nat$ such that $I^{k_V} \mathfrak{G} (V)=0$ (here $k_V$ 
depends upon $\mathfrak{G}$); this follows from Lemma 
\ref{lem:finite_p-gp_nilp} below. 
\end{proof}

\begin{lem}
 \label{lem:finite_p-gp_nilp}
 Let $G$ be a finite $p$-group, and $I G $ be the augmentation ideal of the mod 
$p$ group ring $\fieldp [G]$. Then $I G$ is nilpotent, i.e. there exists $N \in 
\nat$ such that $I^N G =0$.
\end{lem}

\begin{proof}
 This is a standard result. It is proved by induction upon the order of the 
finite $p$-group $G$, the inductive step 
 relying on the fact that the centre of a finite $p$-group is non-trivial.
\end{proof}

\subsection{$p$-finiteness versus finiteness}

The following theorem implies that the two notions of finiteness for 
$\mathfrak{G} \in \ob \fpg$ coincide.

\begin{thm}
\label{thm:pfin_finite}
 Let $\mathfrak{G} \in \ob \fpg$ be a functor to finite $p$-groups such that 
$\mathfrak{G}(0)= \{e \}$. Then $\mathfrak{G}$ is $p$-finite 
if and only if 
  the underlying presheaf $\mathfrak{G}$ in $\prshv$ is finite.
\end{thm}

\begin{proof}
 The implication $\Leftarrow$ is given by Corollary
\ref{cor:finite_generation_implies_finite}. 
 
 For  $\Rightarrow$, since $\fieldp [\mathfrak{G}] \in \f$ is 
coanalytic (by Proposition \ref{prop:gp_ring_coanalytic}), it is possible to 
carry out the argument by passing to unstable (Hopf) algebras. Namely, consider 
the unstable algebra $H:=\kappa \mathfrak{G}$; by Proposition 
\ref{prop:finite-type_kappa}, this has finite type. 
It follows that the group structure of $\mathfrak{G}$ gives $H$ the structure 
of 
a Hopf algebra in $\unstalg$ and, by construction, $\g H \cong \mathfrak{G}$
as $p$-group valued functors.

The proof is by induction upon the length of the $p$-derived series of 
$\mathfrak{G}$, using the correspondence between short exact sequences of Hopf 
algebras and 
short exact sequences of group-valued functors given by Lemma 
\ref{lem:ses_Hopf_group}. If $\Phi \mathfrak{G} = \{e\}$, then $\mathfrak{G}$ 
actually belongs to $\f$ and 
$p$-finite is equivalent to finite. 

For the inductive step, consider the projection $\mathfrak{G} 
\twoheadrightarrow 
\mathfrak{G} / \Phi \mathfrak{G}$. Applying $\kappa$ gives an inclusion of Hopf 
algebras in $\hopfkc$
\[
H' :=  \kappa (\mathfrak{G} / \Phi \mathfrak{G}) \hookrightarrow H
\]
and hence a short exact sequence in $\hopfkc$:
\[
 H' \hookrightarrow H \twoheadrightarrow H\hq H'.
\]
By Lemma \ref{lem:ses_Hopf_group}, one has that $\g (H \hq H') \cong \Phi 
\mathfrak{G}$. Hence, by induction, 
$H'$ and $H\hq H'$ are both $\cala$-finitely generated up to nilpotents. 

Since $H \hq H'$ is $\cala$-finitely generated up to nilpotents, there exists a 
connected unstable algebra $K \subset H$ that is 
$\cala$-finitely generated and such that the composite $K \rightarrow H\hq H'$ 
is an $F$-epimorphism.

  Proposition \ref{prop:hopf_A_fg} provides a sub unstable Hopf algebra $H_K 
\subset H$ that is $\cala$-finitely generated as an unstable algebra and such 
that $K \subset H_K$ (as unstable algebras).
  
 Let $H'' \subset H$ denote the sub Hopf algebra  image in $\hopfkc$ of
 \[
  H' \otimes H_K \rightarrow H.
 \]
By construction, $H''$ is $\cala$-finitely generated and the cokernel $H\hq 
H''$ is nilpotent. 

Applying the functor $\g$ to the short exact sequence of $\hopfkc$ 
\[
 H'' \hookrightarrow H \rightarrow H\hq H'',
\]
by Lemma \ref{lem:ses_Hopf_group} implies that $\g H'' \cong \g H$; since $\g 
(H 
\hq H') = *$.
 (This step relies crucially on having a short exact sequence of groups, not 
just pointed sets.)

This implies that $H$ is $\cala$-finitely generated modulo nilpotents, by 
Corollary 
\ref{cor:equivalent_finite_fgmodnil}. 
\end{proof}

\begin{cor}
 For $H \in \ob \hopfkc$, the underlying unstable algebra of $H$ is 
$\cala$-finitely generated up to nilpotents 
 if and only if $\gamma H$ is $p$-finite. 
 
 If $H$ satisfies the above conditions and $K\subset H$ is a sub Hopf algebra 
in 
unstable algebras, then $K$ is $\cala$-finitely generated up to nilpotents.
\end{cor}

\section{Unstable Hopf algebras modulo nilpotents}

\subsection{Coanalytic presheaves of  $p$-profinite groups}

The material of Section \ref{sect:fun_fpg} leads to the appropriate notion of 
profinite $p$-group valued 
functors.

\begin{nota}
 \ 
 \begin{enumerate}
  \item 
  Let $\pgfin \subset \fpg$ denote the full subcategory of $p$-finite objects. 
  \item 
  For $G\in \ob \gpsc$ that takes values in profinite $p$-groups, let 
$G/\pgfin$ be the full subcategory of 
  $G/ \gpsc$ with objects morphisms $G \rightarrow G'$ with $G' \in \ob \pgfin$ 
that are continuous group morphisms on sections 
  and such that morphisms are commutative diagrams 
  \[
   \xymatrix{
   & G 
   \ar[rd]
   \ar[ld]
   \\
   G' 
   \ar[rr]
   &&
   G'',
   }
  \]
where $G' \rightarrow G''$ is a morphism of $\pgfin$.
\item 
For $G$ as above, set  $G^\omega := 
  \lim_{\substack{\leftarrow \\ G \rightarrow G' \in G/\pgfin}} G',
$ equipped with the canonical map $G\rightarrow G^\omega$. 
\end{enumerate}
\end{nota}

\begin{rem}
Let $G(i)$ be a diagram of $\pgfin$ indexed by a small, cofiltering category 
$\cali$. Then 
 $
  \lim_{\substack{\longleftarrow\\ i \in \cali}} G(i)
 $ 
has underlying presheaf in $\gcoan$ and takes values in $p$-profinite groups. 
 \end{rem}

\begin{defn}
\ 
\begin{enumerate}
 \item 
  For $G\in \ob \gpsc$ which takes values in profinite $p$-groups, say that $G$ 
is $p$-coanalytic if the canonical map $G \rightarrow G^\omega$ is an 
isomorphism. 
 \item 
  Let $\pgpro$ for the category of connected, $p$-coanalytic sheaves and 
continuous group morphisms, equipped with the forgetful functor $\pgpro 
\rightarrow \gpsc$. 
\end{enumerate}
\end{defn}

\begin{rem}
 A $p$-coanalytic presheaf is, in particular, coanalytic. Hence the forgetful 
functor induces $\pgpro \rightarrow (\gcoan)_c$. 
\end{rem}

\begin{lem}\label{lem:g_kappa_Hopf}
\ 
\begin{enumerate}
 \item 
For $H \in \ob \hopfkc$,  $\g H$ is $p$-coanalytic and $\g$ induces a functor 
$\hopfkc \rightarrow \pgpro$. 
\item 
For $G \in \ob \pgpro$ $\kappa G \in \ob \hopfkc$ and $\kappa$ induces a 
functor 
$\pgpro \rightarrow \hopfkc$. 
\end{enumerate}
\end{lem}

\begin{proof}
For the first statement, use the fact that $H$ is the colimit of its sub-Hopf 
algebras in $\unstalg$ that are $\cala$-finitely generated, by Proposition 
\ref{prop:hopf_A_fg}. The second is clear; the colimit of the associated diagram 
in $\hopfkc$ lies in 
$\hopfkc$. The functoriality is clear in both cases. 
\end{proof}

\subsection{Connected unstable Hopf algebras modulo nilpotents}

The analogue of Theorem \ref{thm:HLS_improved} for connected unstable Hopf 
algebras is now essentially tautological.

\begin{nota}
 Let $\hopfkc / \nil$ denote the localization, defined as for $\unstalg/ \nil$ 
(cf. \cite[Part II.1]{HLS} and Section \ref{subsect:unstalg_prelim}).
\end{nota}

\begin{thm}
\label{thm:g_Hopf_fun-pfg}
The functor $\g$ induces an equivalence of categories 
\[
 (\hopfkc/ \nil)\op 
 \stackrel{\cong}{\rightarrow}
 \pgpro.
\]
\end{thm}

\begin{proof}
From the construction and from Lemma \ref{lem:g_kappa_Hopf}, it is clear that 
$\g$ induces a functor  $\hopfkc / \nil \rightarrow \pgpro$. The inverse 
functor 
is induced by $\kappa$. 
\end{proof}

\begin{exam}
 If $H \in \hopfkc$ is  primitively-generated, then $\g H$ is a constant-free 
functor of $\f$ that is dual to a locally finite functor (considered as a 
pro-object). 
\end{exam}

\appendix
\section{Recollections on polynomial functors}
\label{sect:poly}

\subsection{Definitions}

Let $(\calc, 0)$ be a small pointed category, equipped with symmetric monoidal 
structure $(\calc, \odot , 0)$ for which $0$ is the unit. Following Baues and 
Pirashvili \cite{BP}, consider the category of functors 
$\grp^\calc$ from $\calc$ to the category of groups $\grp$, and the associated 
notion of polynomial functor. The trivial group is written $e$. 

\begin{defn}
\label{def:reduced}
 A functor $F \in \ob \grp^\calc$ is  constant-free if $F (0) = e$. 
\end{defn}

The cross-effects of \cite{BP} generalize the Eilenberg-MacLane notion of 
cross-effect from the abelian setting  \cite{EM}. 

\begin{defn}
 \label{def:cross-effects}
For $1 \leq n \in \nat$, let $\cre_n : \grp ^\calc \rightarrow 
\grp^{\calc^{\times n}}$ 
denote the functors defined recursively by 
\begin{enumerate}
 \item 
 $\cre_1 F (C) := \ker \{ F (C) \rightarrow F (0)  \}$;
 \item 
 $\cre_2 F (C, D) := \ker \{ \cre_1 F (C \odot D) \rightarrow \cre_1 F (C) 
\times \cre_1 F (D) \}$ and
 \item 
 $\cre_{n+1}F (C_1, C_2, \ldots, C_{n+1}) := \cre_2\big ( \cre_n  F( - , C_3 , 
\ldots , C_{n+1})\big)(C_1, C_2)$,  for $n \geq 2$.  
\end{enumerate}
(Here the morphisms are induced by the projections to $0$, considered as the 
terminal object of $\calc$.)
 \end{defn}

 The following is standard, and allows usage of $\cre_1$ to be avoided when 
considering constant-free functors. 
 
\begin{lem}
\label{lem:reduced_cre}
 If $F \in \grp^\calc$ is constant-free, then 
 \begin{enumerate}
  \item 
  $\cre_1 F \cong F$; 
  \item 
  the functor $C \mapsto \cre_2 F (C, D) $ is constant-free for any $D \in \ob \calc$.
 \end{enumerate}
\end{lem}

\begin{lem}
 \label{lem:equivalent_cre_reduced}
 If $F \in \ob \grp ^\calc$ is constant-free, then there is a natural isomorphism:
 \[
 \cre_n F (C_1 , \ldots , C_n)
 \cong
 \ker 
 \big\{
 F(C_1 \odot \ldots \odot C_n) 
 \rightarrow 
 \prod_{i=1}^n 
 F (C_1 \odot \ldots \odot \widehat{C_i} \odot \ldots \odot C_n) 
 \big\},
 \]
where $\widehat{C_i}$ indicates that the term is omitted.
\end{lem}

\begin{proof}
 The result is proved by induction on $n$. For $n=2$, by Lemma 
\ref{lem:reduced_cre}, this follows from the Definition of $\cre_2$. 
 For $n>2$, the inductive step proceeds as for the case $n=3$, which is 
\cite[Lemma 1.8]{BP}.
\end{proof}

\begin{defn}
\label{def:poly_general}
 For $n \in \nat$, a functor $F \in \ob \grp ^\calc$ is polynomial of degree 
$\leq n$ if $\cre_{n+1} F= e$ is the constant functor. 
\end{defn}

The following is an immediate consequence of Lemma 
\ref{lem:equivalent_cre_reduced}:

\begin{lem}
\label{lem:equivalent_reduced_poly}
 For $n \in \nat$, a constant-free functor $F \in \ob \grp ^\calc$ is polynomial of 
degree $\leq n$ if and only if the natural transformation
 \[
  F(C_1 \odot \ldots \odot C_{n+1}) 
 \rightarrow 
 \prod_{i=1}^{n+1} 
 F (C_1 \odot \ldots \odot \widehat{C_i} \odot \ldots \odot C_{n+1}) 
 \]
is injective.
\end{lem}

\subsection{Exactness of cross-effects}

\begin{defn}
 For $\mathscr{D}$ a small category, a sequence of functors of 
$\grp^{\mathscr{D}}$, $F_1 \rightarrow F_2 \rightarrow F_3$ is short exact if, 
for all $D \in \ob \mathscr{D}$, 
 \[
  F _1 (D) \rightarrow F_2 (D) \rightarrow F_3 (D)
 \]
is a short exact sequence of groups. 
\end{defn}

The following generalizes the standard result for functors to abelian 
categories:

\begin{prop}
\label{prop:exact_cre}
 Let $1 \leq n \in \nat$. If 
 $
  K 
  \rightarrow 
  G \rightarrow 
  Q 
 $ 
is a short exact sequence of $\grp^\calc$, then 
\[
  \cre_n  K 
  \rightarrow 
 \cre_n G 
  \rightarrow 
 \cre_n Q 
\]
is a short exact sequence of $\grp^{\calc^{\times n}}$.
\end{prop}

\begin{proof}
 The reduction to the case where the functors are constant-free is left to the 
reader. From the recursive definition of cross-effects, it is straightforward to reduce 
 to the case $n=2$. Now, for $F$ constant-free, as in \cite[Section 1]{BP}, 
there is a natural short exact sequence of groups
\[
 \cre_2 F (C, D) 
 \rightarrow 
 F (C \odot D) 
 \rightarrow 
 F(C) \times F(D).
\]
Hence the result follows by the nine (or $3 \times 3$) Lemma in the category 
$\grp$. 
\end{proof}

This immediately provides the {\em thickness} of the polynomial property:

\begin{cor}
\label{cor:thick_poly}
For 
 $
  K 
  \rightarrow 
  G \rightarrow 
  Q 
 $ 
 a short exact sequence of $\grp^\calc$ and $n \in \nat$, 
 $G$ has polynomial degree $\leq n$ if and only if both $K$ and $Q$ have 
polynomial degree $\leq n$. 
\end{cor}

\providecommand{\bysame}{\leavevmode\hbox to3em{\hrulefill}\thinspace}
\providecommand{\MR}{\relax\ifhmode\unskip\space\fi MR }
\providecommand{\MRhref}[2]{%
  \href{http://www.ams.org/mathscinet-getitem?mr=#1}{#2}
}
\providecommand{\href}[2]{#2}

 \end{document}